\documentclass[10pt]{amsart}
\usepackage{amsmath, amsthm, amsfonts, amssymb, graphicx, hyperref, bm,verbatim,mathrsfs,siunitx}
\usepackage{stmaryrd,enumerate}
\theoremstyle{plain}
\newtheorem{thrm}{Theorem}[section]
\newtheorem{lmm}[thrm]{Lemma}
\newtheorem{prpstn}[thrm]{Proposition}
\newtheorem{crllry}[thrm]{Corollary}

\newtheorem*{rmk}{Remark}

\newtheorem*{qstn}{Question}

\numberwithin{sblmm}{thrm} 
\numberwithin{equation}{section}

\DeclareMathOperator*{\vol}{vol}

\newcommand{\Mod}[1]{\ (\mathrm{mod}\ #1)}

\begin{document}
\title{Simultaneous small fractional parts of polynomials}
\author{James Maynard}
\address{Mathematical Institute, Woodstock Road, Oxford OX2 6GG, UK}
\email{james.alexander.maynard@gmail.com}
\begin{abstract}
Let $f_1,\dots,f_k\in\mathbb{R}[X]$ be polynomials of degree at most $d$ with $f_1(0)=\dots=f_k(0)=0$. We show that there is an $n<x$ such that $\|f_i(n)\|_{\mathbb{R}/\mathbb{Z}}\ll x^{c/k}$ for all $1\le i\le k$ for some constant $c=c(d)$ depending only on $d$. This is essentially optimal in the $k$-aspect, and improves on earlier results of Schmidt who showed the same result with $c/k^2$ in place of $c/k$.
\end{abstract}
\maketitle
%
%
%
%
%
%
%
\section{Introduction}
In this paper we consider the following question:
\begin{qstn}
Given $k$ polynomials $f_1,\dots,f_k\in\mathbb{R}[X]$ of degree at most $d$ with $f_1(0)=\dots =f_k(0)=0$, how small can we make the fractional parts $\|f_1(n)\|_{\mathbb{R}/\mathbb{Z}},\dots,\|f_k(n)\|_{\mathbb{R}/\mathbb{Z}}$ over positive integers $n\le x$?
\end{qstn}
Here $\|\cdot\|_{\mathbb{R}/\mathbb{Z}}$ denotes the distance to the nearest integer. Since the polynomial $f(n)=n+1/2$ certainly doesn't attain arbitrarily small fractional parts, it is natural to impose the condition $f_1(0)=\dots=f_k(0)=0$ so that all polynomials individually can obtain small fractional parts. Indeed, it is known that if $f\in\mathbb{R}[X]$ has degree at most $d\ge 2$ and satisfies $f(0)=0$ then the fractional part $\|f(n)\|_{\mathbb{R}/\mathbb{Z}}$ can become arbitrarily small, and the recent work of Baker \cite{Baker} shows that
\begin{equation}
\min_{n\le x}\|f(n)\|_{\mathbb{R}/\mathbb{Z}}\ll_{d}\frac{1}{x^{1/(2d^2-2d)+o(1)}}.
\label{eq:SinglePoly}
\end{equation}
It is worth emphasizing that the bound depends only on $x$ and $d$, and is otherwise completely uniform over all such polynomials $f$. The exponent $1/(2d^2-2d)$ is based on the resolution of Vinogradov's Mean Value Theorem for $k\ge 4$ by Bourgain-Demeter-Guth \cite{Bourgain}, but estimates of the shape $O(1/d^2)$ were known since the work of Wooley \cite{WooleyAnnals} and estimates of the form $O(1/d^2\log{d})$ go back to Vinogradov \cite{Vinogradov}. The bound \eqref{eq:SinglePoly} is certainly not expected to be tight; for monomials the exponent can be improved to $O(1/d\log{d})$, for example, but it is conjectured \cite{BakerBook} that this should be improvable to $1+o(1)$. Unfortunately there currently does not appear to be a feasible approach to make progress on the shape of these exponents with the current techniques.

In the case of $k$ polynomials $f_1,\dots,f_k\in\mathbb{R}[X]$ of degree at most $d$ with $f_1(0)=\dots =f_k(0)=0$, the current record due to Baker and Harman \cite{BakerHarman} is for some $c_d>0$
\begin{equation}
\min_{n\le x}\max_{i\le k}\|f_i(n)\|_{\mathbb{R}/\mathbb{Z}}\ll_{k,d} \frac{1}{x^{1/(k^2+k c_d)+o(1)}}.\label{eq:ManyPolys}
\end{equation}
This refines the initial groundbreaking work of Schmidt \cite{Schmidt} from 1977 who had a similar exponent of the form $1/2k^2$ when $d=2$. A simple argument based on choosing the coefficients of $f_1,\dots,f_k$ uniformly at random shows that one certainly cannot hope to have a result stronger than
\begin{equation}
\min_{n\le x}\max_{i\le k}\|f_i(n)\|_{\mathbb{R}/\mathbb{Z}}\ll \frac{1}{x^{1/k}}.\label{eq:Optimal}
\end{equation}
These and related questions have been the object of a large amount of study in analytic number theory; see \cite{Madritsch,BakerII,BakerIV,BakerV,BakerVI,VaughanWooley,WooleyII,WooleyIII,BakerIII,Schaffer,Zaharescu} for some recent related work. We refer the reader to the book \cite{BakerBook} for a comprehensive overview of these questions.

Our main result is to establish a bound for \eqref{eq:ManyPolys} with an exponent $O_d(1/k)$. By comparing this with \eqref{eq:Optimal} we see that this bound is of the \textit{optimal} shape in the $k$-aspect. This result is new even in the simplest non-linear case when $f_i(n)=\alpha_i n^2$ for $1\le i\le k$ which corresponds to simultaneous Diophantine approximation with squares. More precisely, our main result is the following.
\begin{thrm}\label{thrm:MainTheorem}
Let $k,d$ be positive integers. There is a constant $C_d>2$ depending only on $d$ and a constant $C_{d,k}>2$ depending only on $d$ and $k$ such that the following holds.

Let $f_1,\dots,f_k\in\mathbb{R}[X]$ be polynomials of degree at most $d$ such that $f_1(0)=\dots=f_k(0)=0$.  Let $\epsilon_1,\dots,\epsilon_k\in(0,1/100]$, and put $\Delta=\prod_{i=1}^k\epsilon_i$.

If $\Delta^{-1}\le x^{1/C_d}$ and $x>C_{d,k}$ then there is a positive integer $n<x$ such that
\[
\|f_i(n)\|_{\mathbb{R}/\mathbb{Z}}\le \epsilon_i\qquad\text{for all $i\in\{1,\dots,k\}$}.
\]
\end{thrm}
%
%
%
%
Choosing $\epsilon_1=\dots=\epsilon_k=x^{-1/(kC_d)}$ in Theorem \ref{thrm:MainTheorem} gives the improvement mentioned above. In the language of \cite{BakerVII}, this confirms the conjecture that an arbitrary system of polynomials with $f_1(0)=\dots=f_k(0)=0$ has `Heillbronn status'.
%
%
%
%
\begin{crllry}\label{crllry:MainCor}
Let $f_1,\dots,f_k\in\mathbb{R}[X]$ be polynomials of degree at most $d$ such that $f_1(0)=\dots=f_k(0)=0$. Then there is a positive integer $n<x$ such that
\[
\|f_i(n)\|_{\mathbb{R}/\mathbb{Z}}\ll_{d,k} x^{-c_d/k}\qquad\text{for all $i\in\{1,\dots,k\}$}.
\]
Here $c_d>0$ is a constant depending only on $d$, and the implied constant depends only on $d$ and $k$.
\end{crllry}
%
%
%
%
As with previous works, a noteworthy feature of Theorem \ref{thrm:MainTheorem} and Corollary \ref{crllry:MainCor} is that the result is completely uniform over the coefficients of the polynomials with the implied constants depending only on $d$ and $k$.

The proof as given in this paper would yield a constant $c_d$ in Corollary \ref{crllry:MainCor} which is exponentially small in $d$ ($c_d=10^{-d}$ would probably suffice), but it is likely that with only a small amount of additional effort the constant could be taken to be of the form $c_d=C/d^2$ or perhaps even $C/(d+d^2/k)$ for a relatively small explicit absolute constant $C$. In the interests of emphasizing the main ideas we have chosen not to pursue such explicit bounds in the $d$-aspect. Similarly we have made no effort to control the implied constant's dependence on $d$ or $k$, although it is likely that adapting the ideas behind \cite[Proposition A.2]{GreenTao} would give a reasonable and explicit dependence on $k$ and $d$.
%
%
%
%

%
%
%
%
%
%
%
%
\section{Outline}
%
%
%
%
In the interest of simplicity we consider the case when $f_i(n)=\alpha_i n^2$, since this case still has most of the main features of the problem at hand. As in Schmidt's original work, the argument follows an increment strategy, where either the situation looks `random' or there is additive structure allowing us to pass to a self-similar situation with one fewer polynomial. We obtain improved bounds by getting more structural control over the arithmetic nature of the large Fourier coefficients, allowing for a more complicated but more efficient increment strategy (this successfully achieves the challenge mentioned in \cite[Page 5]{BakerBook} of sharpening Schmidt's `determinant argument').

For a generic choice of $\alpha_1,\dots,\alpha_k$, we expect that the vector of fractional parts $\mathbf{v}(n)=(\|\alpha_1 n^2\|_{\mathbb{R}/\mathbb{Z}},\dots,\|\alpha_k n^2\|_{\mathbb{R}/\mathbb{Z}})$ will equidistribute in the torus $\mathbb{R}^k/\mathbb{Z}^k$. Fourier analysis is well-suited to showing such equidistribution, and one finds that given any intervals $I_1,\dots,I_k$ of length $\delta$ (for some small $\delta>0$) and $x>\delta^{-k-o(1)}$, there is an $n<x$ such that $\mathbf{v}(n)\in I_1\times\dots \times I_k$ unless there is a Diophantine relation
\[
h_1\alpha_1+\dots +h_k\alpha_k\approx \frac{a}{q}
\]
for some constants $h_i\le \delta^{-1-o(1)}$ and some $q<\delta^{-O(k)}$. If $\delta>x^{-c/k}$ for some small $c>0$, such a relation is unusual but would mean that it is genuinely \textit{not} the case that $\mathbf{v}(n)$ equidistributes at this scale. (For example, if $\alpha_1=\alpha_2$, then there is clearly not equidistribution.)

Schmidt \cite{Schmidt} addressed this potential issue by restricting to considering integers $n$ such that $n$ was a multiple of $h_1q$ whenever there is such a Diophantine relation (assuming $h_1\ne 0$, as we may do by relabeling the indices.) For such $n$'s we see that
\[
\| \alpha_1 n^2\|_{\mathbb{R}/\mathbb{Z}}\approx \|\sum_{i=2}^k h_i\alpha_i n^2 /h_1\|_{\mathbb{R}/\mathbb{Z}},
\]
and so if we can find $a_1,\dots,a_k\in\mathbb{Z}$ and $n'<x/q h_1$ such that $|\alpha_i (q h_1n')^2-a_i|<\delta^{-1-o(1)}$ for $2\le i\le k$ and such that $\sum_{i=1}^k h_i a_i=0$ then we can find an $n<x$ such that $\|\alpha_i n^2\|_{\mathbb{R}/\mathbb{Z}}<\delta^{-1}$ for $1\le i\le k$. This essentially reduces the problem of finding $n<x$ such that $\|\alpha_i n^2\|_{\mathbb{R}/\mathbb{Z}}$ is small for $1\le i\le k$  to one of finding $m<x\delta^{k+o(1)}$ such that $\|\alpha_i' m^2\|_{\mathbb{R}/\mathbb{Z}}$ is small for $1\le i\le k-1$ for some reals $\alpha'_1,\dots,\alpha'_{k-1}$. Since the problem is now analogous to the original but with one fewer variable, we may repeat the above procedure $O(k)$ times. We maintain a non-trivial range for $n$ provided $\delta>x^{-c/k^2}$ for some small constant $c>0$, which gives Schmidt's result that there is an $n<x$ such that $\|\alpha_i n^2\|_{\mathbb{R}/\mathbb{Z}}<x^{-c/k^2}$ for some constant $c$ independent of $k$ or $\alpha_1,\dots,\alpha_k$.

The above procedure would produce a bound of size $\|\alpha_i n^2\|_{\mathbb{R}/\mathbb{Z}}<x^{-c/k}$ if at each stage the denominator $q$ was of size $\delta^{-O(1)}$ instead of size $\delta^{-O(k)}$. Therefore let us consider the most problematic case when $q$ is of size $\delta^{-O(k)}\approx x^{c}$. In this case one still has suitable equidistribution via Fourier analysis unless there are $\textit{many}$ vectors $(h_1,\dots,h_k)\in[0,\delta^{-1+o(1)}]$ and coprime integers $a,q$ with $q\in[Q,2Q]$ such that 
\[
h_1\alpha_1+\dots +h_k\alpha_k\approx \frac{a}{q}.
\]
The key new idea in our proof is to exploit the fact we have many such relations rather than just one, and that the $h_i$ must lie in an additively structured set, which will show the rationals $a/q$ cannot have many distinct denominators. In fact, we will show that it must be the case that several of these relations must have the \textit{same} denominator $q$, from which we can reduce to a much lower dimensional situation.

If many of the relations do have the same denominator $q$, then there must be many linearly independent solutions with the same denominator, and so (after relabelling) we can find short vectors  $\mathbf{h}^{(1)}=(h_{1}^{(1)},\dots,h_k^{(1)}),\dots,\mathbf{h}^{(r)}=(h_{1}^{(r)},\dots,h^{(r)}_k)$ in $[0,\delta^{-1+o(1)}]^{k}$ such that $(h^{(1)}_1,\dots,h^{(1)}_r), \dots,(h^{(r)}_1,\dots,h^{(r)}_r)$ are linearly independent in $\mathbb{Z}^r$ and if $q|n$ then
\begin{align*}
\Bigl\| \sum_{i=1}^r h_i^{(1)}\alpha_i n^2\Bigr\|_{\mathbb{R}/\mathbb{Z}}&\approx \Bigl\|\sum_{i=r+1}^k h_i^{(1)}\alpha_i n^2\Bigr\|_{\mathbb{R}/\mathbb{Z}},\\
&\vdots\\
\Bigl\|  \sum_{i=1}^r h_i^{(r)}\alpha_r n^2\Bigr\|_{\mathbb{R}/\mathbb{Z}}&\approx \Bigl\|\sum_{i=r+1}^k h_i^{(r)}\alpha_i n^2\Bigr\|_{\mathbb{R}/\mathbb{Z}}.
\end{align*}
If we also restrict to $\det( h^{(i)}_j)_{1\le i,j\le r}|n$ then we find that, similarly to in Schmidt's argument, we can reduce the problem to a lower dimensional one. In this case, however, we reduce the dimension by $r$ rather than just 1, and in fact we can take $r\gg\log{x}/\log{q}$. In the case when we always have $q \approx x^c$ this process terminates after $O(1)$ iterations rather than $O(k)$ iterations, allowing us to have a maintain a non-trivial range of $n$ if $\delta>x^{-c/k}$ for some small constant $c$.

Alternatively, if there are many different denominators which occur - say $Q^{1/100}$ different denominators of size $Q$ - then it turns out we may find a subset of $Q^{1/200}$ of these denominators which are almost all coprime to one another apart from some fixed integer $d$ which divides all of the denominators in this subset. From this coprimality relation, we see that by adding $r$ of these equations together one finds
\[\Bigl\|\sum_{j=1}^r \sum_{i=1}^k \alpha_i h_i^{(j)}\Bigr\|_{\mathbb{R}/\mathbb{Z}}\approx \Bigl\|\frac{a^{(1)}}{q^{(1)}}+\dots+\frac{a^{(r)}}{q^{(r)}}\Bigr\|_{\mathbb{R}/\mathbb{Z}}\ge \frac{1}{q^{(1)}\cdots q^{(r)}}\]
In particular, we see that almost all combinations of $r/2$ of the equations are distinct, and so there are $\gg Q^{r/200}$ different non-zero combinations. However, we also have that the number of different choices of the coefficients of the $\alpha_i$ in these relations is bounded by $r^k\delta^{-k-o(1)}$. This is less than $Q^{r/200}$ if $r$ is sufficiently large, giving a contradiction. Hence there cannot be many relations with different denominators $q$.

The above sketch is too simplistic in multiple ways - there are quantitative issues if the vectors $\mathbf{h}^{(j)}$ are not essentially orthogonal to each other, and one actually needs to find suitable low height relations to avoid an accumulation of losses through the induction procedure. These can be achieved by exploiting ideas from the geometry of numbers. It is also the case (and was even in Schmidt's original argument) that it is necessary to consider more general approximations by lattice vectors of the vector of polynomials such that the difference lies in a given convex set, which is essentially equivalent to considering approximations of the form $\|f_i(n)\|_{\mathbb{R}/\mathbb{Z}}\le \epsilon_i$ for some given reals $\epsilon_1,\dots,\epsilon_k$.
%
%
%
%
\begin{rmk}
It is interesting to note that the above argument can be interpreted as a density-increment argument in the style of Roth's theorem on arithmetic progressions. This is the first setting which we are aware of when such a strategy produces essentially optimal polynomial-type bounds in a non-trivial situation. Moreover, it has been speculated (see \cite{Gowers}) that even in less structured problems one might hope to have either a small density increment on a `small codimension' set, or a large density increment on a `large codimension' set. We achieve something very much along these lines in this (more structured) setting of fractional parts of polynomials.
\end{rmk}
%
%
%
%
\section{Acknowledgements}
%
%
%
%
J.M. would like to thank Ben Green for introducing him to this question. J.M. was supported by a Royal Society Wolfson Merit Award, and funding from the European Research Council (ERC) under the European Union's     Horizon 2020 research and innovation programme (grant agreement No 851318).
%
%
%
%
\section{Notation}
%
%
%
%
Throughout the paper we assume that we have polynomials $f_1,\dots,f_k\in\mathbb{R}[X]$ of degree at most $d$ with $f_1(0)=\dots=f_k(0)=0$. We let these polynomials be given by $f_i(X)=\sum_{j=1}^d f_{i,j}X^j$. Furthermore, we have reals $\epsilon_1,\dots,\epsilon_k\in(0,1/100]$, and we put $\Delta:=\prod_{i=1}^k\epsilon_i$.

To avoid any confusion about the quantifiers in statements of the form `if $A\ll 1$ then $B\ll 1$, with all implied constants depending only on $d$ and $k$', we emphasize that we take this statement to mean for any positive function $f(d,k)$ of $d$ and $k$, there is a positive function $g_f(d,k)$ depending only on $f$ such that if $|A|\le f(d,k)$ then we have $|B|\le g(d,k)$.
%
%
%
%
\section{The main argument}
%
%
%
%
Our proof of Theorem \ref{thrm:MainTheorem} relies on three key propositions. In this section we show how the theorem follows quickly from the propositions, leaving us with the task of establishing the propositions separately from one another. 

Our first proposition is a standard result which follows from Weyl's bound for polynomial exponential sums.
%
%
%
%
\newpage
\begin{prpstn}[Equidistribution or many linear relations]\label{prpstn:Equidistribution}
Let $f_1,\dots,f_k\in\mathbb{R}[X]$ be polynomials of degree at most $d$ such that $f_1(0)=\dots=f_k(0)=0$. Put $f_i(X)=\sum_{j=1}^d f_{i,j} X^j$. Let $\epsilon_1,\dots,\epsilon_k\in(0,1/100]$, and put $\Delta=\prod_{i=1}^k\epsilon_i$.

Then there is a constant $C_d>0$ depending only on $d$ such that, provided $\Delta^{-C_d}<x$, at least one of the following holds:
\begin{enumerate}
\item We have
\[
\#\{n\le x:\, \|f_i(n)\|_{\mathbb{R}/\mathbb{Z}}<\epsilon_i\,\forall i\}\gg x\prod_{i=1}^k\epsilon_i.
\]
\item There is some $Q\le\Delta^{-C_d}$ such that there are at least $Q^{1/C_d}$ triples $(\mathbf{a},\mathbf{q},\mathbf{h})\in\mathbb{Z}^d\times\mathbb{Z}^d\times\mathbb{Z}^k$ satisfying:
\begin{enumerate}
\item $\gcd(a_j,q_j)=1$ and $1\le q_j\le Q$ for $1\le j\le d$.
\item $h_i\ll \epsilon_i^{-1}\Delta^{-1/(2k)^4}$ for $1\le i\le k$.
\item For each $j\in \{1,\dots,d\}$ we have
\[
\sum_{i=1}^k h_i f_{i,j}=\frac{a_j}{q_j}+O\Bigl(\frac{Q^{C_d}}{x^{j}}\Bigr).
\]
\end{enumerate}
\end{enumerate}
All implied constants depend only on $d$ and $k$.
\end{prpstn}
%
%
%
%
Our second proposition allows us to find structure in the large Fourier coefficients with many of them giving rise to rationals with the same denominator.
%
%
%
%
\begin{prpstn}[Many relations must have the same denominator]\label{prpstn:SameDenom}
Let $f_1,\dots,f_k\in\mathbb{R}[X]$ be polynomials of degree at most $d$ such that $f_1(0)=\dots=f_k(0)=0$. Put $f_i(X)=\sum_{j=1}^d f_{i,j} X^j$. Let $\epsilon_1,\dots,\epsilon_k\in(0,1/100]$, and put $\Delta=\prod_{i=1}^k\epsilon_i$.

Let $C>2$ and $Q\le \Delta^{-C}$ be such that there are at least $Q^{1/C}$ triples $(\mathbf{a},\mathbf{q},\mathbf{h})\in\mathbb{Z}^d\times\mathbb{Z}^d\times\mathbb{Z}^k$ satisfying:
\begin{enumerate}
\item $\gcd(a_j,q_j)=1$ and $1\le q_j\le Q$ for $1\le j\le d$.
\item $h_i\ll \epsilon_i^{-1}\Delta^{-1/(2k)^4}$ for $1\le i\le k$.
\item For each $j\in \{1,\dots,d\}$ we have
\[
\sum_{i=1}^k h_i f_{i,j}=\frac{a_j}{q_j}+O\Bigl(\frac{Q^{C}}{x^{j}}\Bigr).
\]
\end{enumerate}
Then there is a constant $C'_d>0$ depending only on $d$ and $C$ such that provided $\Delta^{-C'_d}<x$ there is some positive integer $q\le Q^{C'_d}$ and at least $Q^{1/C'_d}$ pairs $(\mathbf{a},\mathbf{h})\in\mathbb{Z}^d\times\mathbb{Z}^k$ such that:
\begin{enumerate}
\item $h_i\ll \epsilon_i^{-1}\Delta^{-2/(2k)^4}$ for $i\in\{1,\dots,k\}$.
\item For each $j\in\{1,\dots,d\}$ we have
\[
\sum_{i=1}^k h_i f_{i,j}=\frac{a_j}{q}+O\Bigl(\frac{Q^{C'_d}}{x^{j}}\Bigr).
\]
\end{enumerate}
All implied constants depend only on $d$ and $k$.
\end{prpstn}
%
%
%
%
Our third key proposition allows us to pass from many relations with the same denominator to a reduced system of approximations.
%
%
%
%
\begin{prpstn}[Many relations with the same denominator give rise to a reduced dimension problem]\label{prpstn:ReduceDim}
Let $f_1,\dots,f_k\in\mathbb{R}[X]$ be polynomials of degree at most $d$ such that $f_1(0)=\dots=f_k(0)=0$. Put $f_i(X)=\sum_{j=1}^d f_{i,j} X^j$. Let $\epsilon_1,\dots,\epsilon_k\in(0,1/100]$, and put $\Delta=\prod_{i=1}^k\epsilon_i$.

Let $C>2$  be such that $\Delta^{-1}\le x^{1/4C^2}$, and let $q$ be a positive integer with $q<Q^C$.

Let $\mathcal{S}$ be the set of pairs $(\mathbf{a},\mathbf{h})\in\mathbb{Z}^d\times\mathbb{Z}^k$ such that for $j\in\{1,\dots,d\}$ we have
\[
\Bigl|\sum_{i=1}^k h_if_{i,j}-\frac{a_j}{q}\Bigr|\ll \frac{Q^C}{x^j},
\]
and such that $|h_i|\ll \epsilon_i^{-1}\Delta^{-2/(2k)^4}$. Assume that $\#\mathcal{S}>Q^{1/C}$.

Then there is an integer $k'<k$, polynomials $g_1,\dots,g_{k'}\in\mathbb{R}[X]$ of degree at most $d$ with $g_1(0)=\dots=g_{k'}(0)=0$ and quantities $\epsilon_1',\dots,\epsilon_{k'}'\in(0,1/100]$ and $y<x$ such that:
\begin{enumerate}
\item (Approximations in the new system produce approximations in the old system.) If there is an integer $n'<y$ such that,
\[
\|g_i(n')\|_{\mathbb{R}/\mathbb{Z}}<\epsilon_i'\qquad \text{for all $1\le i\le k'$}
\]
then there is an integer $n<x$ such that
\[
\|f_i(n)\|_{\mathbb{R}/\mathbb{Z}}<\epsilon_i\qquad \text{for all $1\le i\le k$}.
\]
\item (Increased density of approximations.) We have
\[
y(\epsilon_1'\cdots\epsilon_{k'}')^{3C^2-C^2/k'{}^3}\gg x(\epsilon_1\cdots \epsilon_k)^{3C^2-C^2/k^3}.
\]
\end{enumerate}
All implied constants depend only on $k$ and $d$.
\end{prpstn}
%
%
%
%
We see that case (2) of the conclusion of Proposition \ref{prpstn:Equidistribution} satisfies the assumptions of Proposition \ref{prpstn:SameDenom}, and the conclusion of Proposition \ref{prpstn:SameDenom} satisfies the conditions of Proposition \ref{prpstn:ReduceDim}. Thus, putting these three propositions together we obtain
%
%
%
%
\begin{prpstn}[Induction Step]\label{prpstn:MainProp}
Let $d,k$ be positive integers. There is a constant $C_d>2$ depending only on $d$ and $C_{d,k}>2$ depending only on $d$ and $k$ such that the following holds.

Let $f_1,\dots,f_k\in\mathbb{R}[X]$ be polynomials of degree at most $d$ such that $f_1(0)=\dots=f_k(0)=0$. Put $f_i(X)=\sum_{j=1}^d f_{i,j} X^j$. Let $\epsilon_1,\dots,\epsilon_k\in(0,1/100]$, and put $\Delta=\prod_{i=1}^k\epsilon_i$. Let $\Delta^{-1}\le x^{2/C_d}$. 

If there is no positive integer $n<x$ such that
\[
\|f_i(n)\|_{\mathbb{R}/\mathbb{Z}}<\epsilon_i\quad \text{for all $i\in\{1,\dots,k\}$},
\]
then there is a positive integer $k'<k$ and polynomials $g_1,\dots,g_{k'}\in\mathbb{R}[X]$ of degree at most $d$ with $g_1(0)=\dots=g_{k'}(0)=0$ and reals $\epsilon_1',\dots,\epsilon_{k'}'\in(0,1/100]$ and $y\in\mathbb{R}$ with $y<x$ such that both of the following hold:
\begin{enumerate}
\item There is no positive integer $n'<y$ such that 
\[
\|g_i(n')\|_{\mathbb{R}/\mathbb{Z}}<\epsilon_i'\quad\text{for all $i\in\{1,\dots,k'\}$}.
\]
\item We have
\[
y(\epsilon_1'\cdots \epsilon_{k'}')^{C_d(3-1/(k')^3)}\ge \frac{x(\epsilon_1\cdots \epsilon_k)^{C_d(3-1/k^3)}}{C_{d,k}}.
\]
\end{enumerate}

All implied constants depend only on $k$ and $d$.

\end{prpstn}
%
%
%
%
\begin{proof}[Proof of Theorem \ref{thrm:MainTheorem} assuming Proposition \ref{prpstn:MainProp}]
Let $C_d$ and $C_{d,k}$ be the constants of Proposition \ref{prpstn:MainProp}, and let $C_0=\sup_{j\le k}C_{d,j}$ (which depends only on $d$ and $k$). Assume for a contradiction that there is no positive $n<x$ such that $\|f_i(n)\|_{\mathbb{R}/\mathbb{Z}}\le \epsilon_i$ for all $i\in\{1,\dots,k\}$.

We will apply Proposition \ref{prpstn:MainProp} repeatedly to reduce the dimension of the problem we consider. Let us define a \textit{System} to be a tuple $(k,\mathbf{g},\boldsymbol{\delta},y)$ consisting of:
\begin{enumerate}
\item A positive integer $k$.
\item A $k$-tuple $\mathbf{g}$ of real polynomials $(g_1,\dots,g_{k})$ of degree at most $d$ satisfying $g_1(0)=\dots=g_{k}(0)=0$.
\item A $k$-tuple $\boldsymbol{\delta}$ of reals $(\delta_1,\dots,\delta_k)$ with $\delta_i\in(0,1/100]$ for all $i\in\{1,\dots,k\}$.
\item A real $y$ such that there is no positive integer $n<y$ satisfying
\[
\|g_i(n)\|_{\mathbb{R}/\mathbb{Z}}\le \delta_i\qquad \text{for all $i\in\{1,\dots,k\}$}.
\]
\end{enumerate}
Given a System $(k,\mathbf{g},\boldsymbol{\delta},y)$, let $\Delta(\boldsymbol{\delta})=\prod_{i=1}^k\delta_i$. By Proposition \ref{prpstn:MainProp}, if a system $(k_j,\mathbf{g}_j,\boldsymbol{\delta}_j,y_j)$ satisfies $\Delta(\boldsymbol{\delta}_j)^{-1}<y_j^{2/C_d}$ then there is a system $(k_{j+1}, \mathbf{g}_{j+1},\boldsymbol{\delta}_{j+1},y_{j+1})$ such that $k_{j+1}<k_j$, $y_{j+1}\le y_j$ and 
\[
y_{j+1}\Delta(\boldsymbol{\delta}_{j+1})^{C_d(3-1/k_{j+1}^2)}\ge \frac{y_j\Delta(\boldsymbol{\delta}_j)^{C_d(3-1/k_j^2)}}{C_0}.
\] 
In particular, if 
\[
y_j\Delta(\boldsymbol{\delta}_j)^{C_d(3-1/k_j^2)}>C_0^{k_j}
\]
then
\[
y_{j+1}\Delta(\boldsymbol{\delta}_j)^{C_d(3-1/k_{j+1}^2)}>C_0^{k_j-1}\ge C_0^{k_{j+1}}.
\]
Moreover, since $C_d,C_0>2$, this implies that $\Delta(\boldsymbol{\delta}_{j+1})^{-1}<y_{j+1}^{2/C_d}$. Thus, given a system $(k_1,\mathbf{g}_1,\boldsymbol{\delta}_1,y_1)$ with $y_1\Delta(\boldsymbol{\delta}_1)^{C_d(3-1/k_1^2)}>C_0^{k_1}$, we may repeatedly apply Proposition \ref{prpstn:MainProp} to obtain an infinite sequence of systems $(k_j,\mathbf{g}_j,\boldsymbol{\delta}_j,y_j)$ for all $j=1,2,\dots$. But the $k_j$ are a decreasing sequence of positive integers, and so no such sequence can exist. Thus there can be no system $(k_1,\mathbf{g}_1,\boldsymbol{\delta}_1,y_1)$ with $y_1\Delta(\boldsymbol{\delta}_1)^{C_d(3-1/k_1^2)}>C_0^{k_1}$. 

Let us be given a positive integer $k$, a $k$-tuple $\mathbf{f}=(f_1,\dots,f_k)$ of real polynomials of degree at most $d$ with $f_1(0)=\dots=f_k(0)=0$, a $k$-tuple of reals $\boldsymbol{\epsilon}=(\epsilon_1,\dots,\epsilon_k)$ with $\epsilon_i\in(0,1/100]$ for all $i\in\{1,\dots,k\}$ and an real $x$ with $x>C_0^k \Delta(\boldsymbol{\epsilon})^{C_d+1/k^2}$. Then $(k,\mathbf{f},\boldsymbol{\epsilon},x)$ cannot form a system, and so there must be a positive integer $n<x$ such that
\[
\|f_i(n)\|_{\mathbb{R}/\mathbb{Z}}<\epsilon_i\qquad \text{for all $i\in\{1,\dots,k\}$}.
\]
This gives the result.
\end{proof}
%
%
%
%
Since Proposition \ref{prpstn:MainProp} follows immediately from Propositions \ref{prpstn:Equidistribution}, \ref{prpstn:SameDenom} and \ref{prpstn:ReduceDim}, it remains to establish these three propositions. We establish each of these in turn over the next three sections.
%
%
%
%
\section{Initial Fourier analysis and Proposition \ref{prpstn:Equidistribution}}
%
%
%
%
In this section we establish Proposition \ref{prpstn:Equidistribution}. The arguments in this section are standard and well-known to researchers in the field, but for completeness we give complete proofs since the versions we use are slightly different from some occurrences in the literature.
%
%
%
\begin{lmm}[Weyl exponential sum bound]\label{lmm:WeylBound}
Let $f\in\mathbb{R}[X]$ be a monic polynomial of degree $d$, and $\alpha\in\mathbb{R}$ satisfy $\alpha=a/q+O(1/q Q)$ for some $q<Q$. Then there is a constant $c_d>0$ depending only on $d$ such that 
\[
\sum_{n<x}e(f(n)\alpha)\ll \frac{x}{q^{c_d}}+\frac{x}{(x^d/q)^{c_d}}.
\]
The implied constant depends only on $d$.
\end{lmm}
\begin{proof}
This follows from \cite[Lemma 2.4]{Vaughan}.
\end{proof}
%
%
%
%
\begin{lmm}[Modified Weyl exponential sum bound]\label{lmm:ModifiedWeyl}
Let $f(X)=\sum_{i=1}^d f_i X^i\in\mathbb{R}[X]$ be a polynomial of degree $d$ with $f(0)=0$. Then there is a constant $C_d''>2$ depending only on $d$ such that the following holds.

If there is some $Q\in [2,x^{1/C_d''}]$ such that
\[
\Bigl|\sum_{n<x}e(f(n))\Bigr|\ge \frac{x}{Q}
\]
then there are positive integers $q_1,\dots,q_d<Q^{C_d''}$ and integers $a_1,\dots,a_d$ such that $\gcd(a_j,q_j)=1$ for $j\in\{1,\dots,d\}$ and 
\[
f_j=\frac{a_j}{q_j}+O\Bigl(\frac{Q^{C_d''}}{x^j}\Bigr)
\]
for $j\in\{1,\dots,d\}$. The implied constants depend only on $d$.
\end{lmm}
%
%
%
%
\begin{proof}
We prove the result by induction. Assume that for each $j>d-\ell$ we have that the coefficient $f_j$ of $f(X)$ satisfies
\begin{equation}
f_j=\frac{a_j}{q_j}+O\Bigl(\frac{Q^{C_j''}}{x^j}\Bigr)
\label{eq:CoeffApproxs}
\end{equation}
for some coprime integers $a_j,q_j$ with $q_j<Q^{C_j''}$ and some constants $C_{j}''$ bounded only in terms of $d$. In the base case with $\ell=0$ we make no assumption. We wish to show that there is a constant $C_{d-\ell}''$ bounded only in terms of $d$ such that if $Q\le x^{1/C''_{d-\ell}}$ then there are coprime integers $a_{d-\ell},q_{d-\ell}$ with $q_{d-\ell}<Q^{C''_{d-\ell}}$ such that \eqref{eq:CoeffApproxs} holds with $j=d-\ell$. Repeatedly applying this for $0\le \ell<d$ then gives the result since the number of repetitions is also bounded only in terms of $d$.

Let $C\ge\max_{j>d-\ell}C_j''$ be taken sufficiently large in terms of $d$, and let $\tilde{q}=q_d q_{d-1}\cdots q_{d-\ell+1}<Q^{d C}$ (let $\tilde{q}=1$ if $\ell=0$). Since we assume that $Q<x^{1/C''_{d-\ell}}$, we have $Q^{2C}\tilde{q}<x^{1/3}$ on restricting to $C''_{d-\ell}>10 d C$. We can split $\{1,\dots,x\}$ into $O(Q^{2C}\tilde{q})$ disjoint arithmetic progressions with modulus $\tilde{q}$ each containing between $x/Q^{2C}\tilde{q}$ and $2x/Q^{2C}\tilde{q}$ elements. (For each residue class $b\Mod{\tilde{q}}$ greedily take the $\lceil x/Q^{2C}\tilde{q}\rceil$ smallest elements until less than $2\lceil x/Q^{2C}\tilde{q}\rceil$ remain.) Then by the triangle inequality
\[
\Bigl|\sum_{n<x}e(f(n))\Bigr|\ll Q^{2C}\tilde{q}\sup_{\substack{x/Q^{2C}\tilde{q}\le y\le 2x/Q^{2C}\tilde{q}\\ x_0\le x}}\Bigl|\sum_{n<y}e(f(x_0+\tilde{q}n))\Bigr|.
\]
By the hypothesis of the lemma, the left hand side is at least $x/Q$. Thus there must be a choice of integers $y\asymp x/\tilde{q}Q^C$ and $x_0<x$ such that
\begin{equation}
\Bigl|\sum_{n<y}e(f(x_0+\tilde{q}n))\Bigr|\gg \frac{y}{Q}.
\label{eq:FLowerBound}
\end{equation}
From the Diophantine approximations \eqref{eq:CoeffApproxs} and the periodicity of $e(t)$ we see that for $\tilde{q}n<x/Q^{2C}$ we have
\begin{align*}
e\Bigl(\sum_{j>d-\ell}(x_0+\tilde{q}n)^jf_j\Bigr)&=e\Bigl(\sum_{j>d-\ell}f_j x_0^j\Bigr)e\Bigl(\sum_{j>d-\ell}\sum_{i=1}^j\binom{j}{i}f_j\tilde{q}^in^i x_0^{j-i}\Bigr)\\
&=e\Bigl(\sum_{j>d-\ell}f_j x_0^j\Bigr)e\Bigl(O\Bigl(\frac{Q^C\tilde{q}n}{x}\Bigr)\Bigr)\\
&=e\Bigl(\sum_{j>d-\ell}f_j x_0^j\Bigr)+O\Bigl(\frac{1}{Q^C}\Bigr).
\end{align*}
Thus we have
\begin{equation}
\Bigl|\sum_{n<y}e(f(x_0+\tilde{q}n))\Bigr|=\Bigl|\sum_{n<y}e(g(n))\Bigr|+O\Bigl(\frac{y}{Q^C}\Bigr),
\label{eq:GApprox}
\end{equation}
where $g$ is the degree $d-\ell$ polynomial
\[
g(X)=\sum_{i=1}^{d-\ell}(x_0+\tilde{q}X)^if_i.
\]
Taking $C$ sufficiently large in terms of $d$, we see that \eqref{eq:FLowerBound} and \eqref{eq:GApprox} show that
\begin{equation}
\Bigl|\sum_{n<y}e(g(n))\Bigr|\gg \frac{y}{Q}
\end{equation} 
for some $y\asymp x/Q^{2C}\tilde{q}$ and some $x_0$. Let $\alpha=f_{d-\ell}\tilde{q}^{d-\ell}$ be the lead coefficient of $g$. If $\alpha=0$ then \eqref{eq:CoeffApproxs} clearly holds for $j=d-\ell$. Thus we may assume $\alpha\ne 0$. By Dirichlet's Theorem, for any choice of $C'$, there is an approximation
\[
\alpha=\frac{a_{d-\ell}}{q_{d-\ell}}+O\Bigl(\frac{Q^{C'}}{q_{d-\ell}x^{d-\ell}}\Bigr)
\]
for some coprime integers $a_{d-\ell},q_{d-\ell}$ with $q_{d-\ell}<x^{d-\ell}/Q^{C'}$. By applying Lemma \ref{lmm:WeylBound} to the polynomial $g(X)/\alpha$ we see that
\[
\frac{y}{Q}\ll \Bigl|\sum_{n<y}e(g(n))\Bigr|\ll \frac{y}{q_{d-\ell}^{c_d}}+O\Bigl(\frac{y}{Q^{(C'-2 d C)c_d}}\Bigr).
\]
On choosing $C'$ large compared with $c_d$ and $C$, we see that this implies $q_{d-\ell}\ll Q^{C/c_d}<Q^{C'}$. This gives \eqref{eq:CoeffApproxs} with $j=d-\ell$ and $C_{d-\ell}$ large enough in terms of $d$, and so gives the result.
\end{proof}
%
%
%
%
\begin{lmm}[Equidistribution or many large Fourier coefficients]\label{lmm:LargeFourier}
 Let $f_1,\dots,f_k\in\mathbb{R}[X]$ be real valued functions, and $\epsilon_1,\dots,\epsilon_k\in(0,1/2]$ be real numbers, with $\Delta:=\prod_{i=1}^k\epsilon_i$. Then at least one of the following holds:
\begin{enumerate}
\item We have
 \[
\#\{n\le x:\, \|f_i(n)\|_{\mathbb{R}/\mathbb{Z}}\le \epsilon_i\,\forall i\}\gg \Delta x.
\]
\item There is a quantity $Q\ge 2$ such that there are at least $Q^{1/2}$ distinct values of $\mathbf{h}\in \mathbb{Z}^k\backslash\{\mathbf{0}\}$ with $|h_i|<\epsilon_i^{-1}\Delta^{-1/(2k)^4}$ such that
\[
\frac{x}{Q}\le \Bigl|\sum_{n\le x}e\Bigl(\sum_{i=1}^k h_i f_i(n)\Bigr)\Bigr|\le\frac{2x}{Q}.
\]
\end{enumerate}
\end{lmm}
%
%
%
%
\begin{proof}
We fix a smooth function $\phi:\mathbb{R}\rightarrow[0,1]$ with $\phi(t)$ supported on $|t|<1$ which is $1$ on $|t|<1/2$ and let all implied constants depend on $\phi$. Let
\[
\Phi_i(t)=\sum_{m\in \mathbb{Z}}\phi\Bigl(\frac{t+m}{\epsilon_i}\Bigr),
\]
which is clearly 1-periodic, smooth, and supported on $\|t\|_{\mathbb{R}/\mathbb{Z}}<\epsilon_i$. By Poisson summation 
\[
\Phi_i(t)=\epsilon_i\sum_{h\in\mathbb{Z}}\hat{\phi}(\epsilon_i h) e(h f_i(t)).
\]
Since $\phi$ is fixed and smooth, $\phi^{(j)}(t)\ll_j 1$, so $|\hat{\phi}(u)|\ll_j u^{-j}$ for all $j\ge 0$. Thus we see that the terms with $|h|\ge \epsilon_i\Delta^{-1/(2k)^4}$ contribute $O(\Delta^{100})$, and so
\[
\Phi_i(t)=\epsilon_i\sum_{|h|\le \epsilon_i^{-1}\Delta^{-1/(2k)^4}}\hat{\phi}(\epsilon_i h) e(h f_i(t))+O(\Delta^{100}).
\]
Thus we find that (recalling $\hat{\phi}(t)\ll1$)
\begin{align*}
\#\{n\le x:\, &\|f_i(n)\|_{\mathbb{R}/\mathbb{Z}}\le \epsilon_i\forall i\}\ge \sum_{n\le x}\prod_{i=1}^k \Phi_i(f_i(n))\\
&=\Delta\sum_{\substack{h_1,\dots h_k\\ |h_i|<\epsilon_i^{-1}\Delta^{-1/(2k)^4}}}\Bigl(\prod_{i=1}^k \hat{\phi}(\epsilon_i h_i) \Bigr) \sum_{n\le x}e\Bigl(\sum_{i=1}^k h_i f_i(n)\Bigr)+O(x\Delta^{99})\\
&=x\Delta\hat{\phi}(0)^k+O\Bigl(\Delta\sum_{\substack{\mathbf{h}\in \mathbb{Z}^k\backslash\{\mathbf{0}\}\\ |h_i|<\epsilon_i^{-1}\Delta^{-1/(2k)^4}}}\Bigl|\sum_{n\le x}e\Bigl(\sum_{i=1}^k h_i f_i(n)\Bigr)\Bigr|\Bigr)+O(x\Delta^{99}).
\end{align*}
For $\Delta$ sufficiently small we see that $\Delta\hat{\phi}(0)^k+O(\Delta^{99})\gg \Delta$, and so either
\[
\#\{n\le x:\, \|f_i(n)\|_{\mathbb{R}/\mathbb{Z}}\le \epsilon_i\forall i\}\gg \Delta x
\]
or
\[
\sum_{\substack{\mathbf{h}\in \mathbb{Z}^k\backslash\{\mathbf{0}\}\\ |h_i|<\epsilon_i^{-1}\Delta^{-1/(2k)^4}}}\Bigl|\sum_{n\le x}e\Bigl(\sum_{i=1}^k h_i f_i(n)\Bigr)\Bigr|\gg x.
\]
In the latter case, by the pigeonhole principle there is some $Q=2^j$ such that there are at least $Q^{1/2}$ choices of $\mathbf{h}$ in the outer summation such that
\[
\frac{x}{Q}\le \Bigl|\sum_{n\le x}e\Bigl(\sum_{i=1}^k h_i f_i(n)\Bigr)\Bigr|\le\frac{2x}{Q}.
\]
This gives the result.
\end{proof}
%
%
%
%
\begin{proof}[Proof of Proposition \ref{prpstn:Equidistribution}]
Assume that conclusion (1) of Proposition \ref{prpstn:Equidistribution} does not hold, so that we wish to establish conclusion (2). By Lemma \ref{lmm:LargeFourier}, there is a parameter $Q_1$ such that there are $Q_1^{1/2}$ choices of $\mathbf{h}$ for which the corresponding exponential sum is large (of size $\gg x/Q_1$). Since the total number of choices of $\mathbf{h}$ is $O(\Delta^{-1-k/(2k)^4})$, we must have $Q_1\ll \Delta^{-2-2k/(2k)^4}$, and so if $\Delta^{-1}<x^{1/B}$ for $B$ sufficiently large in terms of $d$, then $Q_1<x^{1/C_d''}$. We can then apply Lemma \ref{lmm:ModifiedWeyl}, which shows that each of these values of $\mathbf{h}=(h_1,\dots,h_k)$ then gives rise to a linear equation
\[
\sum_{i=1}^k h_if_{i,j}=\frac{a_j}{q_j}+O\Bigl(\frac{Q_1^{C_d}}{x^j}\Bigr)
\]
with $(a_j,q_j)=1$ and $q_j\le Q_1^{C_d''}$. Letting $Q=Q_1^{C_d''}$ and taking $C_d$ sufficiently large compared with $C_d''$ then gives the result.
\end{proof}
%
%
%
%
\section{Structure in the large Fourier coefficients and Proposition \ref{prpstn:SameDenom}}
%
%
%
%
In this section we prove Proposition \ref{prpstn:SameDenom} by showing many different linear relations with small denominators must give rise to several relations with the same denominator.
%
%
%
%
\begin{lmm}[Expansion or same denominators]\label{lmm:Expansion}
Let $\delta\in(0,1/200)$ and $r$ a positive integer. Let $Q>0$ be large enough in terms of $\delta$ and $r$, and let $\mathcal{S}\subset\mathbb{Z}\times\mathbb{Z}\times\mathbb{Z}^k$ be a set of triples $(a,q,\mathbf{h})$ with $\gcd(a,q)=1$ and $q\le Q$ such that $\#\mathcal{S}\ge Q^\delta$. Then one of the following holds:
\begin{enumerate}
\item There is a $q_0\le Q$ such that at least $\#\mathcal{S}^{1/2}$ of the triples $(a,q,\mathbf{h})\in\mathcal{S}$ have $q=q_0$.
\item The set
\[
\mathcal{A}=\Bigl\{\frac{a_1}{q_1}+\dots+\frac{a_r}{q_r}:\,\text{there exists $\mathbf{h}_1,\dots,\mathbf{h}_r\in\mathbb{Z}^k$ s.t. }(a_i,q_i,\mathbf{h}_i)\in\mathcal{S}\text{ for $1\le i\le r$}\Bigr\}
\]
has cardinality at least $\#\mathcal{S}^{r/5}$.
\end{enumerate}
\end{lmm}
%
%
%
%
\begin{proof}
Throughout the Lemma we will assume that $Q$ is large enough in terms of $\delta$ and $r$ without further comment. We first restrict our attention to a suitable subset of the $q$'s appearing in $\mathcal{S}$. For $j=0,1,\dots$ let 
\[
\mathcal{B}_j=\Bigl\{q\in [2^j,2^{j+1}):\,\text{$\exists\, (a,\mathbf{h})\in\mathbb{Z}\times\mathbb{Z}^k$ with $\gcd(a,q)=1$ and $(a,q,\mathbf{h})\in\mathcal{S}$}\Bigr\}.
\]
Clearly $\mathcal{B}_j$ is empty if $j>2\log{Q}$ since if $(a,q,\mathbf{h})\in\mathcal{S}$ then $q\le Q$. 

If 
\[
\#\{q:\,\exists (a,\mathbf{h})\in\mathbb{Z}\times\mathbb{Z}^k\text{ with }(a,q,\mathbf{h})\in\mathcal{S}\}=\sum_{2^j\le Q}\#\mathcal{B}_{j}\le \#\mathcal{S}^{1/2}
\]
then (by the pigeonhole principle) there is a $q\le Q$ such that there are at least $\#\mathcal{S}^{1/2}$ choices of $(a,\mathbf{h})$ with $(a,q,\mathbf{h})\in\mathcal{S}$, since there are this many on average. Thus condition (1) is satisfied in this case.

Thus we may assume that 
\[
\sum_{j}\#\mathcal{B}_{j}>\#\mathcal{S}^{1/2}\ge Q^{\delta/2},
\]
and so there is some $j_0\le 2\log{Q}$ such that $\#\mathcal{B}_{j_0}>\#\mathcal{S}^{2/5}$. Note that we must have $2^{j_0}>Q^{\delta/3}$ from the trivial bound $\#\mathcal{B}_j\le 2^{j}$.

If there is an integer $d$ which divides at least $\#\mathcal{B}_{j_0}/d^{\delta/10}$ elements of $\mathcal{B}_{j_0}$, we restrict our attention to this subset. By performing this repeatedly, we may assume that there is a fixed integer $d_0$ and a set $\mathcal{B}_{j_0}'\subseteq\mathcal{B}_{j_0}$ such that $\#\mathcal{B}_{j_0}'\ge \#\mathcal{B}_{j_0}/d_0^{\delta/10}$, all elements of $\mathcal{B}'_{j_0}$ are a multiple of $d_0$, and there is no integer $\ell>1$ such that at least $\#\mathcal{B}_{j_0}'/\ell^{\delta/10}$ elements of $\mathcal{B}'_{j_0}$ are a multiple of $d_0\ell$. Since we must have $d_0\le 2^{j_0+1}\le 2Q$, we see that $\#\mathcal{B}_{j_0}'\ge \#\mathcal{S}^{2/5}/Q^{\delta/10}\ge \#\mathcal{S}^{1/4}$. Since $\mathcal{B}_{j_0}'\subseteq \{b\in [2^{j_0},2^{j_0+1}):\,d_0|b\}$ a set of size $O(2^{j_0}/d_0)$, we see this also implies that $d_0<2^{j_0}/Q^{\delta/4}$. 
%
%
%
Finally, we let
\[
\mathcal{B}=\{b:d_0b\in\mathcal{B}'_{j_0}\},
\]
and note that $\mathcal{B}\subseteq [B,2B)$ where we have set $B:=2^{j_0}/d_0$. The above discussion implies that $\#\mathcal{B}\ge \#\mathcal{S}^{1/4}$, that $B\in [\#\mathcal{S}^{1/4},Q]$, and that there is no integer $\ell>1$ such that $\ell$ divides at least $\#\mathcal{B}/\ell^{\delta/10}$ elements of $\mathcal{B}$.

We now wish to show that if we fix a choice of integers $a(b)$ for $b\in\mathcal{B}$ satisfying $\gcd(a(b),d_0 b)=1$, then as $(b_1,\dots,b_r)$ varies in $\mathcal{B}^r$, many of the sums 
\[
\frac{a(b_1)}{d_0b_1}+\dots+\frac{a(b_r)}{d_0b_r}
\]
have different denominators when written as a single fraction in reduced terms, and so in particular many of the expressions are distinct. If all of $d_0,b_1,\dots,b_r$ were pairwise coprime then the denominator would be $d_0b_1\cdots b_r$, and by the divisor bound there are few different choices of $b_1,\dots,b_r$ which would give the same denominator. Instead we are in the situation where the $b_i$'s are `close' to coprime, since we expect they typically have small $\gcd$s by construction of $\mathcal{B}$.

Consider the graph $G=(\mathcal{V},\mathcal{E})$ where the vertex set $\mathcal{V}$ is taken to be $\mathcal{B}$, and the edge set $\mathcal{E}$ is defined by
\[
\mathcal{E}=\{(b_1,b_2)\in \mathcal{B}^2:\,\gcd(b_1,b_2)\ge B^{\delta^2/r^2}\}.
\]
We consider separately two cases.\\

\textbf{Case 1: $\#\mathcal{E}\ge \#\mathcal{V}^2/10r^2$}.\\
In this case there are many pairs with a gcd of some size. If we pick a vertex $v$ in $G$ at random, then the expected number of vertices connected to $v$ is at least $\#\mathcal{V}/10 r^2$, and so (by the pigeonhole principle) there is some $b_0\in \mathcal{B}$ such that there are at least $\#\mathcal{B}/10r^2$ elements $b\in \mathcal{B}$ with $\gcd(b,b_0)>B^{\delta^2/r^2}$. Since there are at most $B^{o(1)}$ divisors of $b_0$, there must be a divisor $d>B^{\delta^2/r^2}$ such that $d|b$ for at least $\#\mathcal{B}/(10r^2 B^{o(1)})>\#\mathcal{B}/d^{\delta/10}$ elements $b\in \mathcal{B}$. But this contradicts the fact that $\mathcal{B}$ is constructed to have no such integers. Thus we must instead have $\#\mathcal{E}<\#\mathcal{V}^2/10r^2$.

\textbf{Case 2: $\#\mathcal{E}<\#\mathcal{V}^2/10r^2$}. \\
In this case the edge density is small, and so a large number of pairs have a very small $\gcd$. If we pick $r$ distinct vertices in $G$ uniformly at random, then the expected number of edges between these vertices is less than $1/9$. In particular, the probability that there are no edges between any of the $r$ chosen vertices is at least $8/9$ (by Markov's inequality). Thus, if we define
\[
\mathcal{C}=\Bigl\{(b_1,\dots,b_r)\in \mathcal{B}^r:\, \gcd(b_i,b_j)<B^{\delta^2/r^2}\text{ for $1\le i< j\le r$}\Bigr\},
\]
then $\#\mathcal{C}\gg_r \#\mathcal{B}^r$. 

We now consider the possible denominators of rationals of the form $a_1/b_1+...+a_r/b_r$ where $(b_1,\dots,b_r)\in\mathcal{C}$. Given $(b_1,\dots,b_r)\in\mathcal{C}$, let
\begin{align*}
\mathcal{R}(b_1,\dots,b_r)&=\Bigl\{(b_1',\dots,b_r')\in\mathcal{C}:\,\exists \,a_1,\dots,a_r,a_1',\dots,a_r'\text{ s.t. }\\
&\gcd(a_i,d_0b_i)=\gcd(a_i',d_0b_i')=1\,\forall i, \frac{a_1}{b_1}+\dots+\frac{a_r}{b_r}=\frac{a_1'}{b_1'}+\dots+\frac{a_r'}{b_r'}\Bigr\}.
\end{align*}
We note that for any choice of $a_1,\dots,a_r$ with $\gcd(a_i,b_i)=1$, the denominator of $a_1/b_1+\dots+ a_r/b_r$ is a multiple of $p^\ell$ if $p^\ell$ divides exactly one of $b_1,\dots,b_r$ and $p^\ell$ divides none of them. Let $\gcd(b,p^\infty)$ denote the largest power of $p$ dividing $b>1$, and $\gcd(b_i,b_j,p^\infty)$ the largest power of $p$ dividing both $b_i$ and $b_j$. We now define
\[
g_p:=\frac{\prod_{i=1}^k \gcd(b_i,p^\infty)}{\prod_{1\le i<j\le r}\gcd(b_i,b_j,p^\infty)^2}
\]
We see that $g_p\le p^\ell$ if $p^\ell$ divides exactly one of $b_1,\dots,b_r$ and $p^{\ell+1}$ divides none of them. Similarly, $g_p\le 1$ if $p^\ell$ divides at least $2$ of the $b_i$ but $p^{\ell+1}$ divides none of them. (If $b_j$ maximizes $\gcd(b_j,p^\infty)$, then $\gcd(b_j,b_i,p^\infty)=\gcd(b_i,p^\infty)$.) Taking the product over all $p$, we see that for any choice of $a_1,\dots,a_r$ with $(a_i,b_i)=1$, the denominator of $a_1/b_1+\dots+a_r/b_r$ must be of size at least $\prod_p g_p$. However, if $(b_1,\dots,b_r)\in\mathcal{C}$ then all pairwise $\gcd$'s are small. Therefore, (regardless of $a_1,\dots,a_r$) the denominator must be of size at least
\[
\prod_p g_p=\frac{\prod_{i=1}^r b_i}{\prod_{1\le i<j\le r}\gcd(b_i,b_j)^2}\ge B^{r-2\delta^2}.
\]
Moreover, any such denominator is clearly of size $O(B^r)$. Thus, given $(b_1,\dots,b_r)\in\mathcal{C}$, there are $O(B^{2\delta^2})$ possible denominators for $a_1/b_1+\dots +a_r/b_r$. Given such a denominator $q>B^{r-2\delta^2}$, if the denominator of $a_1'/b_1'+\dots+a_r'/b_r'$ is also equal to $q$ then $q$ must divide $\prod_{i=1}^r b'_i$. Thus there are $O(B^{2\delta^2})$ such choices of $\prod_{i=1}^rb_i'\ll B^r$ given $q$, and so $O(B^{2\delta^2+o(1)})$ choices of $b_1',\dots,b_r'$ (using the divisor bound). Hence for any choice of $(b_1,\dots,b_r)\in\mathcal{C}$ there are at most $B^{5\delta^2}$ choices of $(b_1',\dots,b_r')$ in total, and so $\#\mathcal{R}(b_1,\dots,b_r)\le B^{5\delta^2}$.

For each $b\in \mathcal{B}$, let $a(b)$ be an integer coprime to $d_0b$ such that $(a(b),d_0b,\mathbf{h})\in\mathcal{S}$ for some $\mathbf{h}$. (This exists from the definition of $\mathcal{B}$.) We now note that given $(b_1,\dots,b_r)\in\mathcal{C}$ the rational $a(b_1)/d_0b_1+\dots+ a(b_r)/d_0b_r$ occurs for at most $\mathcal{R}(b_1,\dots,b_r)$ other elements of $\mathcal{C}$. Thus 
\begin{align*}
\#\mathcal{A}&\ge \#\Bigl\{\frac{a(b_1)}{d_0b_1}+\dots+\frac{a(b_r)}{d_0b_r}:\,(b_1,\dots,b_r)\in\mathcal{C}\Bigr\}\\
&\ge \sum_{(b_1,\dots,b_r)\in\mathcal{C}}\frac{1}{\#\mathcal{R}(b_1,\dots,b_r)}\\
&\ge \frac{\#\mathcal{C}}{B^{5\delta^2}}\ge\frac{\#\mathcal{B}^r}{Q^{6\delta^2}}.
\end{align*}
Recalling that $r\ge 1>200\delta$ and $\#\mathcal{B}>\#\mathcal{S}^{1/4}\ge Q^{\delta/4}$, this gives condition (2), as required.
\end{proof}
%
%
%
%
\begin{lmm}[Many linear relations must have the same denominator]\label{lmm:SameDenom}
Let $\delta\in(0,1/200)$, let $k\ge 2$ a positive integer and let $\alpha_1,\dots,\alpha_k\in[0,1)$. Let $Q$ be large enough in terms of $\delta$ and $k$, and let $\epsilon_1,\dots\epsilon_k\in (0,1]$ be such that $\Delta=\prod_{i=1}^k\epsilon_i$ satisfies $Q^{10\delta}\le \Delta^{-1}\le Q^{(2k)^4}$.

Let $\mathcal{S}\subset\mathbb{Z}\times\mathbb{Z}\times\mathbb{Z}^k$ be a set of triples $(a,q,\mathbf{h})$ with $\gcd(a,q)=1$, $q\le Q$ and $h_i\le \epsilon_i^{-1}\Delta^{-1/(2k)^4}$ such that $\#\mathcal{S}\ge Q^\delta$ and such that if $(a,q,\mathbf{h})\in\mathcal{S}$ then
\[
\Bigl|h_1\alpha_1+\dots+h_k\alpha_k-\frac{a}{q}\Bigr|\le\Bigl(\prod_{i=1}^k\epsilon_i\Bigr)^{100/\delta}.
\]

Then there is a $q_0\le Q$ such that at least $\#\mathcal{S}^{1/2}$ of the triples $(a,q,\mathbf{h})\in\mathcal{S}$ have $q=q_0$.
\end{lmm}
%
%
%
%
\begin{proof}
Choose an integer $r$ such that $Q^{\delta r/20}>\prod_{i=1}^k\epsilon_i^{-1}\ge Q^{\delta r/30}$. We see that such an $r$ must exist and satisfy $r\in [20,30(2k)^4/\delta]$ from our bounds on $\prod_{i=1}^k\epsilon_i^{-1}$ in terms of $Q$. In particular, we may assume that $Q$ is sufficiently large in terms of $r$. If $(a_1,q_1,\mathbf{h}_1),\dots,(a_,q_r,\mathbf{h}_r)\in\mathcal{S}$ then we have for $1\le j\le r$
\[
\sum_{i=1}^k\alpha_i(\mathbf{h}_j)_i=\frac{a_j}{q_j}+O\Bigl(\prod_{i=1}^k\epsilon_i\Bigr)^{100/\delta}=\frac{a_j}{q_j}+O(Q^{-10 r/3}).
\]
Adding these together (and recall that $Q$ is sufficiently large so $r Q^{-r/3}\le 1$ ) gives
\[
\frac{a_1}{q_1}+\dots+\frac{a_r}{q_r}+O(Q^{-3 r})=\sum_{i=1}^k\alpha_i\tilde{h}_i,
\]
where $\tilde{h}_i=\sum_{j=1}^r(\mathbf{h}_j)_i$. Since the denominator of $a_1/q_1+\dots+a_r/q_r$ when written as a single fraction is at most $Q^r$, we see that this fraction is uniquely determined by the integers $\tilde{h}_1,\dots,\tilde{h}_k$, since it is the best rational approximation to $\sum_{i=1}^k\alpha_i\tilde{h}_i$ with denominator at most $Q^r$. But $|\tilde{h}_i|\ll r\epsilon_i^{-1}\Delta^{-1/(2k)^4}$, and so we find
\begin{align*}
&\#\Bigl\{\frac{a_1}{q_1}+\dots+\frac{a_r}{q_r}:\,\exists \mathbf{h}_1,\dots,\mathbf{h}_r\text{ s.t. }(a_1,q_1,\mathbf{h}_1),\dots,(a_r,q_r,\mathbf{h}_r)\in\mathcal{S}\Bigr\}\\
&\le \#\{(\tilde{h}_1,\dots,\tilde{h}_r)\in\mathbb{Z}^r:\,|\tilde{h}_i|\ll r\epsilon_i^{-1}\Delta^{-1/(2k)^4}\}\\
&\le \Bigl(\prod_{i=1}^k\epsilon_i^{-1}\Bigr)^{2}\\
&\le Q^{\delta r/10}.
\end{align*}
Here we used the fact that $r\le \delta^{-1}\log(\prod_{i=1}^k\epsilon_i^{-1})$ and $\prod_{i=1}^k\epsilon_i^{-1}\ge Q^{10\delta}$ can assumed to be sufficiently large in terms of $\delta$ and $k$.

We see that our situation satisfies all the hypotheses of Lemma \ref{lmm:Expansion}, but the above bound is incompatible with the bound of case (2) in Lemma \ref{lmm:Expansion}, since in our situation case (2) would imply that
\[
\#\Bigl\{\frac{a_1}{q_1}+\dots+\frac{a_r}{q_r}:\,\exists \mathbf{h}_1,\dots,\mathbf{h}_r\text{ s.t. }(a_1,q_1,\mathbf{h}_1),\dots,(a_r,q_r,\mathbf{h}_r)\in\mathcal{S}\Bigr\}\ge Q^{\delta r/5}.
\]
Thus, case (1) of Lemma \ref{lmm:Expansion} must hold; there must be a $q_0\le Q$ such that at least $\#\mathcal{S}^{1/2}$ of the triples $(a,q,\mathbf{h})\in\mathcal{S}$ have $q=q_0$.
\end{proof}
%
%
%
%
\begin{lmm}[Many systems of linear relations must have the same denominators]\label{lmm:SameDenom2}
Let $\delta\in(0,1/200)$, let $k,d\ge 2$ be positive integers and let $\alpha_{i,j}\in[0,1)$ for $1\le i\le k$, $1\le j\le d$ be reals. Let $Q$ be large enough in terms of $\delta,d$ and $k$, and let $\epsilon_1,\dots\epsilon_k\in (0,1]$ be such that $\Delta=\prod_{i=1}^k\epsilon_i$ satisfies $Q^{10\delta}\le \Delta^{-1}\le Q^{(2k)^4}$.

Let $\mathcal{S}\subset\mathbb{Z}^d\times\mathbb{Z}^d\times\mathbb{Z}^k$ be a set of triples $(\mathbf{a},\mathbf{q},\mathbf{h})$ satisfying:
\begin{enumerate}
\item $\gcd(a_j,q_j)=1$ and $q_j\le Q$ for $j\in\{1,\dots,d\}$.
\item $h_i\le \epsilon_i^{-1}\Delta^{-1/(2k)^4}$ for $i\in\{1,\dots,k\}$.
\item $\#\mathcal{S}\ge Q^{2^d  \delta}$ .
\item For each $j\in \{1,\dots,d\}$ we have
\[
\Bigl|h_1\alpha_{1,j}+\dots+h_k\alpha_{k,j}-\frac{a_j}{q_j}\Bigr|\le\Bigl(\prod_{i=1}^k\epsilon_i\Bigr)^{100/\delta}.
\]
\end{enumerate}
Then there is a $\mathbf{q}_0\in\mathbb{Z}^d$ such that at least $\#\mathcal{S}^{1/2^d}$ of the triples $(\mathbf{a},\mathbf{q},\mathbf{h})\in\mathcal{S}$ have $\mathbf{q}=\mathbf{q}_0$.
\end{lmm}
%
%
%
%
\begin{proof}
This follows from $d$ applications of Lemma \ref{lmm:SameDenom}. Given $\mathcal{S}'\subseteq\mathcal{S}$, for $j\in \{1,\dots,d\}$ let
\[
\pi_j(\mathcal{S}')=\{(a,q,\mathbf{h}):\,\exists\,\mathbf{a},\mathbf{q}\text{ s.t. }(\mathbf{a},\mathbf{q},\mathbf{h})\in\mathcal{S}'\text{ and }a_j=a,\, q_j=q\}.
\]
 We note that
\[
\Big(\prod_{i=1}^k\epsilon_i\Bigr)^{100/\delta}\le Q^{-10},
\]
and so given $\mathbf{h}\in\mathbb{Z}^k$ there is at most one choice of $\mathbf{a},\mathbf{q}\in\mathbb{Z}^d$ such that $(\mathbf{a},\mathbf{q},\mathbf{h})\in\mathcal{S}$, since $a_j/q_j$ is the best rational approximation with denominator at most $Q$ to $h_1\alpha_{1,j}+\dots+h_k\alpha_{k,j}$ if there is any $\mathbf{a},\mathbf{q}$ such that $(\mathbf{a},\mathbf{q},\mathbf{h})\in\mathcal{S}$ (recall that we must have $\gcd(a_j,q_j)=1$ if $(\mathbf{a},\mathbf{q},\mathbf{h})\in\mathcal{S}$). In particular, $\#\pi_j(\mathcal{S}')=\#\mathcal{S}'$ for all $j$ for any set $\mathcal{S}'\subseteq \mathcal{S}$. 

Given $\mathcal{S}'\subseteq\mathcal{S}$, let $\ell_j(\mathcal{S}')$ be an integer maximizing
\[
\#\{(a,\mathbf{h}):\,(a,\ell,\mathbf{h})\in\pi_j(\mathcal{S})\}
\]
over all choices of $\ell\in\mathbb{Z}$ (if there are multiple possibilities we make an arbitrary choice of one). If $\#\pi_j(\mathcal{S}')\ge Q^{\delta}$, then by Lemma \ref{lmm:SameDenom} at least $\#\pi_j(\mathcal{S}')^{1/2}$ triples $(a,q,\mathbf{h})\in\pi_j(\mathcal{S}')$ have $q=\ell_{j,\mathcal{S}'}$. We now let $\mathcal{S}_0=\mathcal{S}$, and define $\mathcal{S}_1\supseteq\dots\supseteq\mathcal{S}_d$ in turn, by 
\[
\mathcal{S}_j:=\{(\mathbf{a},\mathbf{q},\mathbf{h})\in\mathcal{S}_{j-1}:\,q_j=\ell_j(\mathcal{S}_{j-1})\}.
\]
Since $\#\mathcal{S}\ge Q^{2^d\delta}$, we see that $\#\mathcal{S}_{j}\ge \#\mathcal{S}_{j-1}^{1/2}\ge Q^{2^{d-j}\delta}\ge Q^\delta$ for $j\in\{1,\dots,d\}$ by repeatedly applying Lemma \ref{lmm:SameDenom}. In particular, we have $\#\mathcal{S}_d\ge \#\mathcal{S}^{1/2^d}$. Finally, we note that $\mathcal{S}_d$ is the set of triples $(\mathbf{a},\mathbf{q},\mathbf{h})\in\mathcal{S}$ such that 
\[
\mathbf{q}=\mathbf{q}_0:=(\ell_{1}(\mathcal{S}_0),\ell_{2}(\mathcal{S}_1),\dots,\ell_{d}(\mathcal{S}_{d-1})),
\]
and so we have the result.
\end{proof}
\begin{proof}[Proof of Proposition \ref{prpstn:SameDenom}]
By assumption, there is some $Q\le(\prod_{i=1}^k\epsilon_i)^{-C_d}$ such that there are at least $Q^{1/C_d}$ triples $(\mathbf{a},\mathbf{q},\mathbf{h})\in\mathbb{Z}^d\times\mathbb{Z}^d\times\mathbb{Z}^k$ with $\gcd(a_j,q_j)=1$ and $q_j\le Q$ for $j\in\{1,\dots,d\}$, and  with $h_i\ll \epsilon_i^{-1}\Delta^{-1/(2k)^4}$ for $i\in\{1,\dots,k\}$, and with
\[
\sum_{i=1}^k h_if_{i,j}=\frac{a_j}{q_j}+O\Bigl(\frac{Q^{C_d}}{x^{j}}\Bigr).
\]
If $Q\le \Delta^{-1/(2k)^4}$ then we just take one such triple $(\mathbf{a},\mathbf{q},\mathbf{h})$. In this case the triples $(j\mathbf{a},j\mathbf{q},j\mathbf{h})$ for $j\in\{1,\dots,Q\}$ then give $Q$ relations of the desired type provided $C_d'>C_d+1$, since $jh_i\ll Q\epsilon_i^{-1}\Delta^{-1/(2k)^4}\ll \epsilon_i^{-1}\Delta^{-2/(2k)^4}$. Thus we may assume that $Q^{(2k)^4}>\prod_{i=1}^k\epsilon_i^{-1}$.

We now apply Lemma \ref{lmm:SameDenom2} with $\delta=(10\cdot 2^d\cdot C_d)^{-1}$. Provided $C_d'>100/\delta+C_d^2$ we see the bounds $Q\le (\prod_{i=1}^k\epsilon_i^{-1})^{C_d}$ and $(\prod_{i=1}^k\epsilon_i^{-1})^{C'_d}<x$ imply that we have $Q^{C_d}/x^d<(\prod_{i=1}^k\epsilon_i)^{100/\delta}$, and all so the hypotheses of Lemma \ref{lmm:SameDenom2} are satisfied. This shows that there is a $\mathbf{q}_0\in\mathbb{Z}^d$ such that at least $Q^{\delta}$ of the triples $(\mathbf{a},\mathbf{q},\mathbf{h})$ have $\mathbf{q}=\mathbf{q}_0$. Thus these all give rise to a rational $a_j/q$ where $q=\prod_{i=1}^d(\mathbf{q}_0)_i$. This gives the result.
\end{proof}
%
%
%
%
\section{Dimension reduction via geometry of numbers and Proposition \ref{prpstn:ReduceDim}}
%
%
%
%
In this section we prove Proposition \ref{prpstn:ReduceDim} using estimates from the geometry of numbers, thereby completing the proof of Theorem \ref{thrm:MainTheorem}.
%
%
%
%
\begin{lmm}[Many relations give rise to orthogonal generators]\label{lmm:Orthogonal}
Let $\eta>0$ be sufficiently small in terms of $k$ and $d$. Let $B_1,\dots,B_k>1$ satisfy $\prod_{i=1}^kB_i\le \eta^{-1/2}$, and $\beta_{i,j}\in \mathbb{R}$ for $1\le i\le k$, $1\le j\le d$. Let $\mathcal{R}$ be the region in $\mathbb{R}^{k+d}$ defined by
\begin{align*}
\mathcal{R}=\Bigl\{(h_1,\dots,h_k,a_1,\dots,a_d)\in\mathbb{R}^{k+d}:&\,\Bigl|\sum_{i=1}^k h_i \beta_{i,j}-a_j\Bigr|\le \eta^j\text{ for }1\le j\le d,\\
&\qquad|h_i|\le B_i\text{ for $1\le i\le k$}\Bigr\},
\end{align*}
and assume that $\#(\mathcal{R}\cap\mathbb{Z}^{k+d})= N$, with $N$ sufficiently large in terms of $k$ and $d$. 

Then there is an integer $r\in \{1,\dots,k\}$ and vectors $\mathbf{h}^{(1)},\dots,\mathbf{h}^{(r)}\in\mathbb{Z}^k$ and $\mathbf{a}^{(1)},\dots,\mathbf{a}^{(r)}\in\mathbb{Z}^d$ such that:
\begin{enumerate}
\item (The $\mathbf{h}^{(j)},\mathbf{a}^{(j)}$ are a system of Diophantine approximations.) For each $j\in\{1,\dots,r\}$ the vector $(h^{(j)}_1,\dots,h^{(j)}_k,a_1^{(j)},\dots,a_d^{(j)})$ lies in $\mathcal{R}\cap\mathbb{Z}^{k+d}$.
\item (The $\mathbf{h}^{(j)}$ are quasi-orthogonal after rescaling.) Let $\tilde{h}^{(j)}_i=h^{(j)}_i/B_i$ for $1\le j\le r$, $1\le i\le k$. Then we have
\[
\|\tilde{\mathbf{h}}^{(1)}\wedge\dots\wedge\tilde{\mathbf{h}}^{(r)}\|\asymp \|\tilde{\mathbf{h}}^{(1)}\|_\infty\cdots \|\tilde{\mathbf{h}}^{(r)}\|_\infty.
\]
\item (The $\mathbf{h}^{(j)}$ generate many elements of $\mathcal{R}\cap\mathbb{Z}^{k+d}$). Let $\tilde{\mathbf{h}}^{(j)}$ be as above. We have
\[
\|\tilde{\mathbf{h}}^{(1)}\|_\infty\cdots \|\tilde{\mathbf{h}}^{(r)}\|_\infty\ll \frac{1}{N^{1/(d+1)}}.
\]
\end{enumerate}
All implied constants depend at most on $k$ and $d$.
\end{lmm}
%
%
%
%
We recall that $\|\mathbf{h}^{(1)}\wedge\dots\wedge\mathbf{h}^{(r)}\|$ is the $r$-dimensional volume of the parallelepiped formed by the vectors $\mathbf{h}^{(1)},\dots,\mathbf{h}^{(r)}$, which is the (Euclidean) length of the vector in $\mathbb{R}^{\binom{k}{r}}$ of all determinants of $r\times r$ submatrices of the $r\times k$ matrix with columns $\mathbf{h}^{(1)},\dots,\mathbf{h}^{(r)}$.
%
%
%
%

\begin{proof}
After potentially permuting the $B_i$ and $\beta_{i,j}$, we may assume without loss of generality that $B_1\ge B_2\ge \dots \ge B_k$.  Let $\Lambda\subset \mathbb{R}^{k+d}$ be the lattice
\[
\mathbb{Z}\Bigl(\frac{1}{B_1}\mathbf{e}_1-\sum_{j=1}^d\frac{\beta_{1,j}}{\eta^j}\mathbf{e}_{k+j}\Bigr)+\dots+\mathbb{Z}\Bigl(\frac{1}{B_k}\mathbf{e}_k-\sum_{j=1}^d\frac{\beta_{k,j}}{\eta^j}\mathbf{e}_{k+j}\Bigr)+\mathbb{Z}\frac{1}{\eta}\mathbf{e}_{k+1}+\dots +\mathbb{Z}\frac{1}{\eta^d}\mathbf{e}_{k+d},
\]
where $\mathbf{e}_1,\dots,\mathbf{e}_{k+d}$ are the standard basis vectors of $\mathbb{Z}^{k+d}$. We see that elements of $\mathcal{R}\cap\mathbb{Z}^{k+d}$ correspond to elements of $\Lambda$ with all components bounded by 1 in absolute value. By standard lattice theory, there is a basis $\mathbf{b}_1,\dots,\mathbf{b}_{k+d}$ of $\Lambda$ (see \cite[Lemma 4.1]{Polys}, for example) such that for any $n_1,\dots,n_{k+d}\in\mathbb{Z}$ we have
\[
\Bigl\|\sum_{i=1}^{k+d}n_i\mathbf{b}_i\Bigr\|_\infty\asymp \sum_{i=1}^{k+d} |n_i| \|\mathbf{b}_i\|_\infty\asymp \sum_{i=1}^{k+d} |n_i|\lambda_i
\]
where $0<\lambda_1\le \lambda_2\le \dots \le \lambda_{k+d}$ are the successive minima of $\Lambda$ and the implied constants depend only on $k$ and $d$. 

If $\lambda_1\le 1/N^{1/(d+1)}$ then the conclusion of the lemma is satisfied with $r=1$ and $h^{(1)}_i=(\mathbf{b}_1)_i B_i$ for $1\le i \le k$ and $a_j^{(1)}=\eta (\mathbf{b}_1)_{k+j}+\sum_{i=1}^k (\mathbf{b}_1)_i\beta_{i,j}B_i$ for $1\le j\le d$. (Such a vector $\mathbf{h}^{(1)}$ is non-zero since if $\mathbf{b}^{(1)}$ was 0 in the first $k$ coordinates it must have norm at least $1/\eta>1$.) We see this choice satisfies the conclusion of the lemma. Thus we may assume that $\lambda_1>1/N^{1/(d+1)}$.

Recall that $B_1\cdots B_k<\eta^{-1/2}$ and $\lambda_1\cdots \lambda_{k+d}\asymp \det(\Lambda)$. We have that
\begin{align*}
\frac{1}{\det(\Lambda)}&=\vol\Bigl\{\mathbf{t}\in\mathbb{R}^{d+k}:\,\Bigl\|\sum_{i=1}^k\Bigl(\frac{t_i}{B_i}\mathbf{e}_i+\sum_{j=1}^d\frac{t_i\beta_{i,j}}{\eta^j}\mathbf{e}_{k+j}\Bigr)-\sum_{i=1}^{d}\frac{t_{k+i}}{\eta^i}\mathbf{e}_{k+i}\Bigr\|_\infty\le 1\Bigr\}\\
&=B_1\cdots B_k\eta^{d(d+1)/2}\\
&\le \eta^{d^2/2}.
\end{align*}
In particular, $\lambda_1\cdots \lambda_k\gg \eta^{-d/2}>1$, and so $\lambda_j>1$ for some $j$ (since $\eta$ is sufficiently small in terms of $k$ and $d$). Since $\mathcal{R}\cap\mathbb{Z}^{k+d}\ne\emptyset$, we also have that $\lambda_1\le 1$. Thus there must be some integer $J$ such that $\lambda_{J}\le 1< \lambda_{J+1}$. 
We note that for some suitably large constant $C$ (depending only on $k$ and $d$)
\begin{align*}
\{\mathbf{x}\in \Lambda:\,\|\mathbf{x}\|_{\infty}<1\}&=\Bigl\{\sum_{i=1}^{k+d}n_i\mathbf{b}_i:\,(n_1,\dots,n_{k+d})\in\mathbb{Z}^{k+d},\,\|\sum_{i=1}^{k+d}n_i\mathbf{b}_i\|_\infty\le 1\Bigr\}\\
&\subseteq \Bigl\{\sum_{i=1}^{k+d}n_i\mathbf{b}_i:\,(n_1,\dots,n_{k+d})\in\mathbb{Z}^{k+d},\,\sum_{i=1}^{k+d}|n_i|\|\mathbf{b}_i\|_\infty\le C\Bigr\}\\
&\subseteq \Bigl\{\sum_{i=1}^{k+d}n_i\mathbf{b}_i:\,(n_1,\dots,n_{k+d})\in\mathbb{Z}^{k+d},\,|n_i|\le C\lambda_i^{-1}\Bigr\}.
\end{align*}
The final set on the right hand side has cardinality
\[
\ll \prod_{i=1}^k\Bigl(1+\frac{C}{\lambda_i}\Bigr)\ll \frac{1}{\lambda_1\cdots \lambda_J}.
\]
Thus we have $\lambda_1\cdots \lambda_J\ll N^{-1}$. Since $\lambda_1\ge N^{-1/(d+1)}$ we have $\lambda_1\cdots\lambda_d\ge N^{-d/(d+1)}\gg \lambda_1\cdots \lambda_J$. Thus we see that $J>d$.

The determinant of $\Lambda$ is given by the determinant of the $(k+d)\times(k+d)$ matrix $M_1$ with columns $\mathbf{b}_1,\dots,\mathbf{b}_{k+d}$, and satisfies $\det(\Lambda)=\det(M_1)\asymp \|\mathbf{b}_1\|_\infty\cdots\|\mathbf{b}_{k+d}\|_\infty\asymp\lambda_1\cdots\lambda_{k+d}$. This implies that some $J\times J$ submatrix $M_2$ of the $J\times(k+d)$ matrix with columns $\mathbf{b}_1,\dots,\mathbf{b}_J$ has $\det(M_2)\asymp \lambda_1\cdots\lambda_J$. (If all such submatrices had determinant bounded by $\delta\lambda_1\cdots \lambda_J$, then by expanding the determinant of $M_1$ into a sum of such determinants, and using $\|\mathbf{b}_j\|_{\infty}\le \lambda_j$, we see the determinant of $M_1$ would be $O_k(\delta\lambda_1\cdots\lambda_{k+d})$, contradicting our lower bound if $\delta$ is sufficiently small in terms of $k$ and $d$). Similarly, we see that there is some choice of $\mathcal{I}=\{i_1,\dots,i_{J-d}\}\subseteq\{1,\dots,J\}$ such that the $(J-d)\times(J-d)$ submatrix $M_\mathcal{I}$ of $M_2$ formed by removing the final $d$ rows and removing the $i^{th}$ column for each $i\in\{1,\dots,J\}\setminus\mathcal{I}$ from $M_2$ satisfies
\[
\det(M_\mathcal{I})\gg \prod_{i\in\mathcal{I}}\lambda_i.
\]
(Consider expanding the determinant of $M_2$ via the bottom $d$ rows so that it is a sum of $O(1)$ of such determinants with the coefficient of $\det(M_\mathcal{I})$ of size $\ll \prod_{i\notin\mathcal{I}}\|\mathbf{b}_i\|_\infty\ll \prod_{i\notin\mathcal{I}}\lambda_i$.)

Let $\mathbf{b}_i'\in\mathbb{R}^{k}$ be the vector formed by removing the last $d$ coordinates of $\mathbf{b}_i$ for $1\le i\le k+d$. The above discussion implies that 
\[
\|\mathbf{b}_{i_1}'\wedge\dots\wedge\mathbf{b}_{i_{J-d}}'\|\ge \det(M_{\mathcal{I}})\gg \prod_{i\in\mathcal{I}}\lambda_i \ge \prod_{i\in\mathcal{I}}\|\mathbf{b}_i'\|_\infty,
\]
since one of the $(J-d)\times(J-d)$ submatrices formed from taking $J-d$ rows from $\mathbf{b}_{i_1}',\dots,\mathbf{b}_{i_{J-d}}'$ is $M_{\mathcal{I}}$. (Recall that $M_{\mathcal{I}}$ was formed from $M_2$ be removing the final $d$ rows, and so cannot contain the row corresponding to the final $d$ coordinates of the $\mathbf{b}_i$.) Finally, on recalling that $N^{-1/(d+1)}\le \lambda_1\le \dots\le \lambda_J$ and $\lambda_1\cdots \lambda_J\ll N^{-1}$, we have
\[
\prod_{i\in\mathcal{I}}\|\mathbf{b}'_i\|_\infty\ll \|\mathbf{b}_{i_1}'\wedge\dots\wedge\mathbf{b}_{i_{J-d}}'\|\ll \prod_{i\in\mathcal{I}}\lambda_i\ll N^{d/(d+1)}\lambda_{1}\cdots \lambda_J\ll \frac{1}{N^{1/(d+1)}}.
\]
We now have the result of the lemma on putting $r=J-d$ and taking
\begin{align*}
\mathbf{h}^{(j)}&=(B_{1}(\mathbf{b}_{i_j})_1,\dots,B_{k}(\mathbf{b}_{i_j})_k)\in\mathbb{Z}^k,\\
a_\ell^{(j)}&=\eta (\mathbf{b}_{i_j})_{k+\ell}+\sum_{i=1}^k (\mathbf{b}_{i_j})_i\beta_{i,\ell}B_i\text{ for }1\le \ell\le d,
\end{align*}
 for $1\le j\le r=J-d$. Note that this choice has $\tilde{\mathbf{h}}^{(j)}=\mathbf{b}_{i_j}'$.
\end{proof}
%
%
%
%
\begin{lmm}\label{lmm:Lattices}
Let $H_1$ be a $r\times r$ invertible integer matrix and $H_2$ an $r\times \ell$ integer matrix. Let $\Lambda_1,\Lambda_2\subseteq\mathbb{Z}^r$ and $\Lambda_3\subseteq\mathbb{Z}^\ell$ be lattices defined by
\begin{align*}
\Lambda_1&=H_1(\mathbb{Z}^r),\\
\Lambda_2&=H_1(\mathbb{Z}^r)+H_2(\mathbb{Z}^{\ell}),\\
\Lambda_3&=\{\mathbf{y}\in\mathbb{Z}^{\ell}:\,\exists\,\mathbf{x}\in\mathbb{Z}^r\text{ s.t. }H_1\mathbf{x}=H_2\mathbf{y}\}.
\end{align*}
Then
\[
\det(\Lambda_1)=\det(\Lambda_3)\det(\Lambda_2).
\]
\end{lmm}
%
%
%
%
\begin{proof}
We first note that $\Lambda_1,\Lambda_2,\Lambda_3$ are all full rank since $H_1$ has non-zero determinant. Let $\Lambda_1$ have determinant $D_1=\det(H_1)$ and $\Lambda_2$ have determinant $D_2$. Since $\Lambda_1\subseteq \Lambda_2\subseteq\mathbb{Z}^r$, $D_1$ and $D_2$ are integers with $D_2|D_1$. Since $D_1\mathbb{Z}^r\subseteq\Lambda_1$, we have that $D_1\mathbb{Z}^{\ell}\subseteq\Lambda_3$. We also have that $D_1\mathbb{Z}^{r}\subseteq D_2\mathbb{Z}^r\subseteq\Lambda_2$. Thus, letting $[D_1]$ denote $\{1,\dots,D_1\}$, we have 
\[
\frac{D_1^r}{\det(\Lambda_2)}=\#\{\mathbf{z}\in[D_1]^r:\,\exists \,\mathbf{x}\in[D_1]^r,\mathbf{y}\in[D_1]^{\ell}\text{ s.t. }H_1\mathbf{x}-H_2\mathbf{y}=\mathbf{z}\Mod{D_1}\}.
\]
 The number of representations $r(\mathbf{z})$ of $\mathbf{z}$ as $H_1\mathbf{x}-H_2\mathbf{y}$ with $\mathbf{x}\in[D_1]^r$, $\mathbf{y}\in[D_1]^\ell$ is either 0 (if $\mathbf{z}\notin\Lambda_2$) or equal to $r(\mathbf{0})$ (if $\mathbf{z}\in\Lambda_2$) by linearity. Thus, since $\sum_{\mathbf{z}\in[D_1]^r}r(\mathbf{z})=D_1^{\ell+r}$, we have
\[
\frac{D_1^r}{\det(\Lambda_2)}=\frac{D_1^{\ell+r}}{r(\mathbf{0})}.
\]
But then we have that for any given $\mathbf{y}\in[D_1]^\ell$, the number $r_2(\mathbf{y})$ of $\mathbf{x}\in [D_1]^r$ such that $H_1\mathbf{x}=H_2\mathbf{y}\pmod{D_1}$ is either $0$ (if $H_2\mathbf{y}\notin\Lambda_1$) or is equal to $r_2(\mathbf{0})$ (if $H_2\mathbf{y}\in\Lambda_1$). Thus
\begin{align*}
r(\mathbf{0})&=\#\{\mathbf{y}\in[D_1]^{\ell}:\,\exists\,\mathbf{x}\in[D_1]^r\text{ s.t. }H_1\mathbf{x}=H_2\mathbf{y}\pmod{D_1}\}\cdot r_2(\mathbf{0})\\
&=\frac{D_1^{\ell}}{\det(\Lambda_3)}\cdot \#\{\mathbf{x}\in[D_1]^r:\,H_1\mathbf{x}=\mathbf{0}\pmod{D_1}\}\\
&=\frac{D_1^{\ell}}{\det(\Lambda_3)}\cdot \frac{D_1^r}{[D_1\mathbb{Z}^r:\Lambda_1]}.
\end{align*}
But $\det(\Lambda_1)=D_1=[\Lambda_1:\mathbb{Z}^r]$, and $[D_1\mathbb{Z}^r:\mathbb{Z}^r]=D_1^r$ so $[D_1\mathbb{Z}^r:\Lambda_1]=D_1^{r}/\det(\Lambda_1)$. Thus we have
\[
\frac{D_1^r}{\det(\Lambda_2)}=\frac{D_1^{\ell+r}}{r(\mathbf{0})}=\det(\Lambda_3)\cdot [D_1\mathbb{Z}^r:\Lambda_1]=D_1^r\frac{\det(\Lambda_3)}{\det(\Lambda_1)}.\qedhere
\]
\end{proof}
%
%
%
%
\begin{lmm}[Orthogonal relations give rise to reduced dimension problem]\label{lmm:ReduceDim}
Let $C>2$ and $f_1,\dots,f_k\in\mathbb{R}[X]$ be polynomials of degree at most $d$ with $f_1(0)=\dots=f_k(0)=0$. Put $f_i(X)=\sum_{j=1}^df_{i,j}X^j$.

Let $B_1,\dots,B_k\ge 1$ and $q_0\in \mathbb{Z}_{>0}$. Let $\eta\in[0,1/100]$ be such that
\[
\eta<\frac{q_0^C}{x}.
\]
Let $r\in \{1,\dots,k\}$ and $\mathbf{h}^{(1)},\dots,\mathbf{h}^{(r)}\in\mathbb{Z}^k$ and $\mathbf{a}^{(1)},\dots,\mathbf{a}^{(r)}\in\mathbb{Z}^d$ satisfy:
\begin{enumerate} 
\item  ${h}^{(\ell)}_i\le B_i$ for $1\le i\le k$ and $1\le \ell\le r$.
\item For $1\le \ell\le r$ and $1\le j\le d$ we have
\[
\Bigl|\sum_{i=1}^k h^{(\ell)}_i f_{i,j}-\frac{a^{(\ell)}_j}{q_0}\Bigr|\le\eta^j.
\]
\item Put $\tilde{h}^{(\ell)}_i= h^{(\ell)}_i/B_i$ for $1\le \ell \le r$ and $1\le i\le k$. We have
\[
\|\tilde{\mathbf{h}}^{(1)}\wedge\dots\wedge\tilde{\mathbf{h}}^{(r)}\|\asymp \|\tilde{\mathbf{h}}^{(1)}\|_\infty\cdots \|\tilde{\mathbf{h}}^{(r)}\|_\infty,
\]
and
\[
\|\tilde{\mathbf{h}}^{(1)}\|_\infty\cdots \|\tilde{\mathbf{h}}^{(r)}\|_\infty\ll \frac{1}{q_0^{1/C}}.
\]
\end{enumerate}
Then there is an integer $k'<k$, real polynomials $g_1,\dots,g_{k'}\in\mathbb{R}[X]$ of degree at most $d$ with $g_1(0)=\dots=g_k(0)=0$ and quantities $B_1',\dots,B_{k'}'\ge 2$ and $y<x$ such that:
\begin{enumerate}
\item (Approximations in the new system produce approximations in the old system) If there is an integer $n'<y$ such that
\[
\|g_i(n')\|_{\mathbb{R}/\mathbb{Z}}<\frac{1}{B_i'}\qquad \text{for all $1\le i\le k'$}
\]
then there is an integer $n<x$ such that
\[
\|f_i(n)\|_{\mathbb{R}/\mathbb{Z}}<\frac{1}{B_i}\qquad \text{for all $1\le i\le k$}.
\]
\item (Increased density of approximations) We have
\[\frac{y}{(B_1'\cdots B_{k'}')^{3C^2-C^2/k'{}^3}}\gg \frac{x}{(B_1\cdots B_k)^{3C^2-C^2/k^3-C^2/k^4}}.\]
\end{enumerate}
All implied constants depend only on $k$ and $d$.
\end{lmm}
%
%
%
%
\begin{proof}
Let $\tilde{M}$ be the $r\times k$ matrix with rows $\tilde{\mathbf{h}}^{(1)},\dots,\tilde{\mathbf{h}}^{(r)}$, and $M$ be the $r\times k$ matrix with rows $\mathbf{h}^{(1)},\dots,\mathbf{h}^{(r)}$. Let $\tilde{H}_1$ be the $r\times r$ submatrix of $\tilde{M}$ with largest determinant (in absolute value). By permuting the coordinates, (and permuting the $f_i$), we may assume that $\tilde{H}_1$ is formed by taking the first $r$ columns of $\tilde{M}$. Let $\tilde{H}_2$ be the $r\times (k-r)$ matrix formed by taking the final $k-r$ columns of $\tilde{M}$. Similarly, let $H_1$ be the matrix formed by taking the first $r$ columns of $M$, and $H_2$ formed by taking the last $k-r$ columns of $M$.

By the set-up of the lemma, we have for each $j\in \{1,\dots,d\}$
\begin{equation}
\tilde{H}_1\begin{pmatrix} f_{1,j} B_1 \\  \vdots \\ f_{r,j} B_r\end{pmatrix} = -\tilde{H}_2\begin{pmatrix}f_{r+1,j}B_{r+1}\\ \vdots \\ f_{k,j} B_k\end{pmatrix}+\frac{1}{q_0}\begin{pmatrix} a^{(1)}_j\\ \vdots \\ a^{(r)}_j \end{pmatrix} + O(\eta^j).
\label{eq:LinearSystem}
\end{equation}
By construction, $\det(\tilde{H}_1)$ is the largest determinant of any $r\times r$ submatrix of $M$ formed with columns $\tilde{\mathbf{h}}^{(1)},\dots,\tilde{\mathbf{h}}^{(r)}$, and so $\det(\tilde{H}_1)\gg \|\tilde{\mathbf{h}}^{(1)}\wedge\dots\wedge\tilde{\mathbf{h}}^{(r)}\|$. By the assumption of the lemma we have $ \|\tilde{\mathbf{h}}^{(1)}\wedge\dots\wedge\tilde{\mathbf{h}}^{(r)}\|\gg \|\tilde{\mathbf{h}}^{(1)}\|_\infty\cdots\|\tilde{\mathbf{h}}^{(r)}\|_\infty$, and so $\det(\tilde{H}_1)\asymp \|\tilde{\mathbf{h}}_1\|_\infty\cdots\|\tilde{\mathbf{h}}_r\|_\infty$. It follows that $(H^{-1})_{i,j}\ll \|\tilde{\mathbf{h}}_j\|_\infty^{-1}$ for all $1\le i,j\le r$.

We multiply \eqref{eq:LinearSystem} by $n^j$ and then sum over $j\in\{1,\dots,d\}$. Rearranging, we find that for any choice of $b_1,\dots,b_k\in\mathbb{Z}$, we have
\begin{align*}
\begin{pmatrix} (f_1(n)-b_1)B_1 \\  \vdots \\ (f_r(n)-b_r)B_r\end{pmatrix}& = -\tilde{H}_1^{-1}\tilde{H}_2\begin{pmatrix}(f_{r+1}(n)-b_{r+1})B_{r+1}\\ \vdots \\ (f_k(n)-b_k)B_k\end{pmatrix}\\
& + \tilde{H}_1^{-1}\left(H_2\begin{pmatrix}b_{r+1}\\ \vdots\\ b_k\end{pmatrix}- H_1\begin{pmatrix}b_1\\ \vdots\\ b_r\end{pmatrix}+\frac{1}{q_0}\begin{pmatrix} \sum_{j=1}^d a^{(1)}_j n^j\\ \vdots \\ \sum_{j=1}^d a^{(r)}_j n^j \end{pmatrix}\right)\\
&+O\left( (n\eta+n^d\eta^d)\begin{pmatrix}
1/\|\tilde{\mathbf{h}}_1\|_\infty\\ \vdots \\ 1/\|\tilde{\mathbf{h}}_r\|_\infty\end{pmatrix}\right).
\end{align*}
We note that $(\tilde{H}_2)_{i,j}\ll \|\tilde{\mathbf{h}}^{(i)}\|_\infty$ for all $1\le i\le r$ and $1\le j\le k-r$. Recalling that $(\tilde{H}_1^{-1})_{i,j}\ll \|\tilde{\mathbf{h}}^{(j)}\|_\infty^{-1}$ for all $1\le i,j\le r$, we see that all entries of $\tilde{H}_1^{-1}\tilde{H}_2$ are of size $O(1)$. In particular, if $b_1,\dots, b_k$ are such that for each $r+1\le j\le k$
\begin{equation}
|f_j(n)-b_j|\le \frac{\delta}{B_j},
\label{eq:ApproxCondition}
\end{equation}
then we have that 
\[
 -\tilde{H}_1^{-1}\tilde{H}_2\begin{pmatrix}(f_{r+1}(n)-b_{r+1})B_{r+1}\\ \vdots \\ (f_k(n)-b_k)B_k\end{pmatrix}=O(\delta).
\]
If we have $\delta<1$ and
\begin{equation}
n\le \frac{\delta \min_i\|\tilde{\mathbf{h}}_i\|_\infty}{\eta},
\label{eq:SizeCondition}
\end{equation}
then (recalling $\|\tilde{\mathbf{h}}_i\|_\infty\le 1$)
\[
(n\eta+n^d\eta^d)\begin{pmatrix}1/\|\tilde{\mathbf{h}}_1\|_\infty\\ \vdots \\ 1/\|\tilde{\mathbf{h}}_r\|_\infty\end{pmatrix}=O(\delta).
\]
Thus, if \eqref{eq:ApproxCondition} and \eqref{eq:SizeCondition} hold and also we have
\begin{equation}
H_1\begin{pmatrix}b_1\\ \vdots\\ b_r\end{pmatrix}-H_2\begin{pmatrix}b_{r+1}\\ \vdots\\ b_k\end{pmatrix}=\frac{1}{q_0}\begin{pmatrix} \sum_{j=1}^d a^{(1)}_j n^j\\ \vdots \\ \sum_{j=1}^d a^{(r)}_j n^j \end{pmatrix},
\label{eq:LatticeCondition}
\end{equation} 
then we have for each $1\le i\le r$
\[
\|f_i(n)\|_{\mathbb{R}/\mathbb{Z}}\ll \frac{\delta}{B_i}.
\]
In particular, we have $\|f_i(n)\|_{\mathbb{R}/\mathbb{Z}}\le 1/B_i$ for $1\le i \le r$ if $\delta$ is chosen to be a sufficiently small constant (depending only on $k$ and $d$) provided \eqref{eq:LatticeCondition}, \eqref{eq:ApproxCondition} and \eqref{eq:SizeCondition} hold.

Let $\Lambda_1\subseteq\Lambda_2\subseteq\mathbb{Z}^r$ be full-rank lattices defined in terms of the matrices $H_1,H_2$ by
\[
\Lambda_1=H_1(\mathbb{Z}^r),\qquad \Lambda_2=H_1(\mathbb{Z}^r)+H_2(\mathbb{Z}^{k-r}).
\]
(They are both full rank since $H_1$ has non-zero determinant.) Let $\Lambda_1$ have determinant $D_1=\det(H_1)$ and $\Lambda_2$ have determinant $D_2$. Since $\Lambda_1\subseteq\Lambda_2\subseteq\mathbb{Z}^r$, $D_1$ and $D_2$ are integers with $D_2|D_1$. Any sublattice of $\mathbb{Z}^r$ with determinant $D_2$ contains $D_2\mathbb{Z}^r$. Therefore for each $j\in \{1,\dots,d\}$ there exists a choice of $b_{1,j}',\dots,b_{k,j}'\in\mathbb{Z}$ such that
\[
H_1\begin{pmatrix}b_{1,j}'\\ \vdots\\ b_{r,j}'\end{pmatrix}-H_2\begin{pmatrix}b_{r+1,j}'\\ \vdots\\ b_{k,j}'\end{pmatrix}=D_2^j q_0^{j-1}\begin{pmatrix}a^{(1)}_j\\ \vdots \\ a^{(r)}_j \end{pmatrix}.
\]
We restrict our attention to $b_i$ of the form $b_i=\sum_{j=1}^d b_{i,j}'n^j/q_0^j D_2^j+b_i''$ for $1\le i\le k$, where $b_1'',\dots,b_k''\in\mathbb{Z}$ satisfy
\begin{equation}
H_1\begin{pmatrix}b_1''\\ \vdots\\ b_r''\end{pmatrix}-H_2\begin{pmatrix}b_{r+1}''\\ \vdots\\ b_k''\end{pmatrix}=0.
\label{eq:LatticeRestriction}
\end{equation}
To ensure that $b_1,\dots,b_k\in\mathbb{Z}$, we will restrict our consideration to integers $n$ such that $D_2q_0|n$.

The equation \eqref{eq:LatticeRestriction} forces $(b_{r+1}'',\dots,b_k'')$ to lie in a rank $r-k$ lattice $\Lambda_3$, given by
\[
\Lambda_3=\{\mathbf{y}\in\mathbb{Z}^{k-r}:\,\exists\mathbf{x}\in\mathbb{Z}^r\text{ s.t. }H_1\mathbf{x}=H_2\mathbf{y}\}.
\]
By Lemma \ref{lmm:Lattices}, $\Lambda_3$ has determinant $D_1/D_2$.

Let $\mathbf{z}_1,\dots,\mathbf{z}_{r-k}$ be a Minkowski-reduced basis for $\Lambda_3$, so in particular $D_1/D_2=\det(\Lambda_3)\asymp\|\mathbf{z}_1\|_\infty\cdots\|\mathbf{z}_{k-r}\|_\infty$. Let $n=D_2q_0n'$ for some $n'\in\mathbb{Z}$. We see that we can find $(b_1,\dots,b_k)$ satisfying the conditions \eqref{eq:LatticeCondition} and \eqref{eq:ApproxCondition} provided we can find $m_1,\dots, m_{k-r}\in\mathbb{Z}$ such that for $r+1\le i\le k$ we have
\begin{align}
\Bigl|f_i(D_2q_0n')-\sum_{j=1}^d b_{i,j}'(n')^j -\sum_{i=1}^{k-r}m_i(\mathbf{z}_i)_{j-r}\Bigr|&=\Bigl|\tilde{f}_i(n')-\sum_{i=1}^{k-r}m_i(\mathbf{z}_i)_{j-r}\Bigr|\nonumber\\
&\le \frac{\delta}{B_j}.
\label{eq:LatticeApprox}
\end{align}
Here we have set $\tilde{f}_i(X)\in\mathbb{R}[X]$ to be the polynomial
\[
\tilde{f}_i(X)=f_i(D_2q_0X)-\sum_{j=1}^d b'_{i,j}X^j
\]
for each $i\in \{r+1,\dots,k\}$. We note that $\tilde{f}$ had degree at most $d$ and has $\tilde{f}_i(0)=0$.

Let $Z$ be the $(k-r)\times(k-r)$ matrix with columns $\mathbf{z}_1,\dots,\mathbf{z}_{r-k}$. Since $\det(Z)\asymp\|\mathbf{z}_1\|_\infty\cdots\|\mathbf{z}_{r-k}\|_\infty$ we have that $Z^{-1}_{i,j}\ll \|\mathbf{z}_j\|_\infty^{-1}$. Thus we see that if we can find $n'$ and $m_1,\dots,m_{k-\ell}$ such that
\[
\left|Z^{-1}\begin{pmatrix} \tilde{f}_{r+1}(n')\\ \vdots\\\tilde{f}_k(n') \end{pmatrix}-\begin{pmatrix} m_1\\ \vdots \\ m_{k-r}\end{pmatrix} \right|
\le \delta^2 \begin{pmatrix}1/(B_{r+1}\|\mathbf{z}_1\|_\infty) \\ \vdots\\ 1/(B_k\|\mathbf{z}_{r-k}\|_\infty) \end{pmatrix},
\]
then \eqref{eq:LatticeApprox} holds if $\delta$ is a sufficiently small constant (depending only on $k$). We now let $k'=k-r$ and $B_i'=\delta^{-2}B_{r+i}\|\mathbf{z}_i\|_\infty$ and $g_1,\dots,g_{k-r}\in\mathbb{R}[X]$ be given by
\[
 \begin{pmatrix}g_1(n')\\ \vdots\\ g_{k-r}(n') \end{pmatrix}=Z^{-1}\begin{pmatrix} \tilde{f}_{r+1}(n')\\ \vdots\\ \tilde{f}_k(n') \end{pmatrix}.
\]
We see that the $g_i$ are polynomials of degree at most $d$ with $g_i(0)=0$ since the $\tilde{f}_i$ are. Finally, we put $y=\delta x \min_i\|\tilde{\mathbf{h}}_i\|_\infty/(q_0^{C+1}D_2)$, and note that since $\eta<q_0^C/x$ we have $y<\delta \min_i\|\tilde{\mathbf{h}}_i\|_\infty/(\eta q_0D_2)$.

Putting everything together, we see that if there is an $n'<y$ such that
\[
\|g(n')\|_{\mathbb{R}/\mathbb{Z}}\le \frac{1}{B_i'}
\]
for $1\le i\le k'=k-r$, then there is an $n=n' q_0 D_2$ with $n<x$ and $n<\delta\min_i\|\tilde{\mathbf{h}}_i\|_\infty/\eta$ such that
\[
\|f_i(n)\|_{\mathbb{R}/\mathbb{Z}}\le \frac{1}{B_i}
\]
for $1\le i\le k$.

Thus we are left to verify the size estimates with this choice of $B_1',\dots,B_{k-r}'$ and $y$. We have that (recalling that $1/\delta=O(1)$)
\begin{align*}
\prod_{i=1}^{k-r}B_i'&=\frac{\|\mathbf{z}_1\|_\infty\cdots\|\mathbf{z}_{k-r}\|_\infty\prod_{i=r+1}^k B_i}{\delta^{2(k-r)}}\\
&\ll \frac{D_1\prod_{i=r+1}^k B_i }{D_2}.
\end{align*}
This implies that
\begin{align*}
\frac{y}{(\prod_{i=1}^{k-r}B_i')^{C_2}}&=\frac{\delta x \min_i\|\tilde{\mathbf{h}}_i\|_\infty}{q_0^{C+1} D_2(\prod_{i=1}^{k-r}B_i')^{C_2}}\\
&\gg \frac{x D_2^{C_2}\min_i\|\tilde{\mathbf{h}}_i\|_\infty}{ q_0^{C+1} D_1^{C_2}(\prod_{i=r+1}^k B_i )^{C_2}}.
\end{align*}
We recall that $D_1=\det(H_1)\asymp \|\mathbf{h}_1\|_\infty\cdots\|\mathbf{h}_r\|_\infty\ll B_1\cdots B_r/q_0^{1/C}$, and that $\min_i\|\tilde{\mathbf{h}}_i\|_\infty\gg \prod_{i=1}^r\|\tilde{\mathbf{h}}_i\|_\infty\gg  D_1/(B_1\cdots B_r)$.  This gives
\begin{align*}
\frac{y}{(\prod_{i=1}^{k-r}B_i')^{C_2}}&\gg \frac{x}{(\prod_{i=r+1}^k B_i )^{C_2}} \frac{D_2^{C_2-1} }{D_1^{C_2-1}q_0^{C+1} B_1\cdots B_r}\\
&\gg \frac{x}{(\prod_{i=1}^k B_i )^{C_2} } q_0^{(C_2-1)/C-C-1}D_2^{C_2-1}.
\end{align*}
Finally, we choose $C_2=3C^2-C^2/(k-r)^3$. Since $D_2,q_0\ge 1$ and $(C_2-1)/C-C-1>0$, this gives
\[
\frac{y}{(\prod_{i=1}^{k-r}B_i')^{3C^2-C^2/(k-r)^3}}\gg \frac{x}{(B_1\cdots B_r)^{3C^2-C^2/(k-r)^3}}.
\]
Since $k>r\ge 1$ we have $1/(k-r)^3\ge 1/k^3+1/k^4$, which gives the result.
\end{proof}
%
%
%
%
\begin{proof}[Proof of Proposition \ref{prpstn:ReduceDim}]
Lemma \ref{lmm:Orthogonal} (taking $\beta_{i,j}=f_{i,j}$, $\eta=Q^C/x$ and $B_i\ll \epsilon_i^{-1}\Delta^{-2/(2k)^4}$) shows that if the assumptions of Proposition \ref{prpstn:ReduceDim} hold then we can find a subset of essentially orthogonal generators. Using these generators in Lemma \ref{lmm:ReduceDim} then gives the required conclusion. Indeed, the first two claims of the proposition are clear. For the final claim we note that
\begin{align*}
y (\epsilon_1'\cdots \epsilon_{k'}')^{3C^2-C^2/k'{}^3}&=\frac{y}{(B_1'\cdots B_{k'}')^{3C^2-C^2/k'{}^3}}\\
&\gg \frac{x}{(B_1\cdots B_k)^{3C^2-C^2/k^3-C^2/k^4}}\\
&=x(\epsilon_1\cdots \epsilon_k\Delta^{2/(2k)^4})^{(3C^2-C^2/k^3-C^2/k^4)}\\
&\gg x(\epsilon_1\cdots \epsilon_k)^{3C^2-C^2/k^3}.
\end{align*}
This gives the final claim, establishing Proposition \ref{prpstn:ReduceDim}. 
\end{proof}
%
%
%
%
\bibliographystyle{plain}
\bibliography{Polys}

\end{document}